\documentclass[11pt]{article}

\usepackage{float}

\usepackage{amsfonts}
\usepackage{amsthm}
\usepackage{amssymb,amsmath}
\usepackage{latexsym}

\usepackage{graphics}
\usepackage{psfrag}
\usepackage{dsfont}
\usepackage{color}
\usepackage{tikz}

\definecolor{gray5}{gray}{0.8}
\definecolor{gray1}{gray}{0.4}
\definecolor{gray2}{gray}{0.6}
\definecolor{gray3}{gray}{0.7}
\definecolor{gray4}{gray}{0.3}

\numberwithin{equation}{section}

\newtheorem     {thm}{Theorem}[section]

\newtheorem     {lem}[thm]{Lemma}

\newtheorem     {rem}[thm]{Remark}

\begin{document}

\title{Approximations of the allelic frequency spectrum in general supercritical branching populations} 
\date{}
\author{\textsc{Benoit Henry$^{1,2}$}}

\footnotetext[1]{Madynes team, INRIA Nancy -- Grand Est, IECL -- UMR 7503,
  Nancy-Universit\'e, Campus scientifique, B.P.\ 70239, 54506 Vand\oe uvre-l\`es-Nancy Cedex,
  France}

\footnotetext[2]{LORIA -- UMR 7503,
Nancy-Universit\'e, Campus scientifique, B.P.\ 70239, 54506 Vand\oe uvre-l\`es-Nancy Cedex,
  France, E-mail: \texttt{benoit.henry@univ-lorraine.fr}
}
\maketitle

\begin{abstract}
We consider a general branching population where the lifetimes of individuals are i.i.d.\ with arbitrary distribution and where each individual gives birth to new individuals at Poisson times independently from each other. In addition, we suppose that individuals experience mutations at Poissonian rate $\theta$ under the infinitely many alleles assumption assuming that types are transmitted from parents to offspring. This mechanism leads to a partition of the population by type, called the allelic partition. The main object of this work is the frequency spectrum $A(k,t)$ which counts the number of families of size $k$ in the population at time $t$. The process $(A(k,t),\ t\in\mathbb{R}_+)$ is an example of non-Markovian branching process belonging to the class of general branching processes counted by random characteristics. In this work, we propose methods of approximation to replace the frequency spectrum by simpler quantities. Our main goal is study the asymptotic error made during these approximations through central limit theorems.
In a last section, we perform several numerical analysis using this model, in particular to analyze the behavior of one of these approximations with respect to Sabeti's Extended Haplotype Homozygosity \cite{sabeti}.
\end{abstract}          
\bigskip



\section{Introduction}
In this paper, we consider a general branching population where the lifetimes of the individuals and their reproductions processes are i.i.d.\ Moreover, we assume that their lifetimes are distributed according to an arbitrary probability distribution $\mathbb{P}_V$ and that the births occur, during their lifetime, according to a Poisson process with rate $b$. The tree underlying this dynamics is called a splitting tree. This class of random trees was introduced in \cite{GK} by Geiger and Kersting and has been widely studied in the last decade \cite{L10,L11,L13}.

We suppose, in addition, that mutations occur on individuals and that each new mutation confers to its holder a brand new type (i.e. never seen in the population): this is the  \emph{infinitely many alleles} assumption.
This allows modeling the occurrence of a new type in a population (such as a new species or a new phenotype in a given species). We also suppose that every individual inherits the type of its parent. This model leads to a partition of the population by types. The frequency spectrum of the population alive at time $t$ is defined as the sequence of number $\left(A(k,t)\right)_{k\geq1}$ where, for each $k$, $A(k,t)$ is the number of families of size $k$ in the population. The famous example of Ewens sampling formula gives explicit expression for the law of the frequency spectrum \cite{EV} when the genealogy is given by the Kingman's coalescent.  Other works studied similar quantities in the case of Galton-Waston branching processes (see \cite{Ber} or \cite{Gri}). In our model, the frequency spectrum has also been widely studied in the past \cite{CL1,CL2,CLR,CH}.

Another object of interest is the process $(N_{t},\ t\in\mathbb{R}_{+})$ which counts the number of living individuals in the population at a given time $t$. This process is known as binary homogeneous Crump-Mode-Jagers process. One of the main result of the theory of such process is the law of large number which gives in our particular case that $e^{-\alpha t }N_{t}$ converges almost surely to a random variable $\mathcal{E}$ which is exponential conditionally on non-extinction (for some positive constant $\alpha$).

As for $e^{-\alpha t }N_{t}$, it is also known that the quantities $e^{-\alpha t}A(k,t)$ converge almost surely to $c_{k}\mathcal{E}$, where $c_{k}$ is an explicit constant. This result can be easily obtained by conjunction of the works of \cite{CL1} and \cite{Rich} using the theory of general branching processes counted by random characteristics (a complete statement can be found in \cite{CLR}). An alternative proof avoiding the use of the general branching processes theory can be found in \cite{CH}.

It appears that the frequency spectrum $(A(k,t))_{k\geq 1}$ is a quantity which is hard to manipulate from the probabilistic point of view (see \cite{CL1,CL2,CH}). This implies that such a model is inconvenient for practical applications. In this work we propose to use the laws of large numbers in order to replace $(A(k,t))_{k\geq 1}$ by more manipulable quantities and propose to investigate the error made during this approximation. The first possible approximation is the following.

\textbf{Approximation 1:}
\[
(A(k,t))_{k\geq 1}\approx (c_{k})_{\geq 1}e^{\alpha t}\mathcal{E}.
\]

However, this is unsatisfactory for practical applications since the random variable $\mathcal{E}$ is not observable at finite times. Another idea is to exploit the fact that the random variable appearing in the law of large numbers for $A(k,t)$ and for $N_{t}$ is the same. This leads to the second approximation.

\textbf{Approximation 2:}
\[
A(k,t)\approx  (c_{k})_{\geq 1}N_{t}.
\]
In order to investigate the errors made during this approximation (at least asymptotically), one would like to have central limit theorems associated to the law of large numbers for the frequency spectrum.
 In a previous work \cite{H1}, we have showed that the error in the convergence of $e^{-\alpha t }N_{t}$ is of order $e^{\alpha t/2}$ and obtained a central limit theorem for this error. An important aspect of the method introduced in \cite{H1} is that it can be used to derive CLTs for other branching processes counted by random characteristics. In particular, the main goal of this work is to obtain central limit theorems for the convergence of the frequency spectrum.
We also study the Markovian cases (when $\mathbb{P}_{V}$ is exponential) where we can obtain more information on the limit distribution.

The original motivation of this study (and of other works on this model \cite{CL1,CL2,CH}) comes from the works of Sabeti and al. \cite{sabeti} where the frequency spectrum is used to detecte positive selection of an allele in an increasing population. More specifically, suppose that you want to detect the positive selection of an allele on a given gene. The main idea is that, under neutral evolution, the allele under consideration needs a long time to reach a high frequency in the population. Hence, if the frequency of the allele w.r.t.\ its age is significantly higher than the expected frequency (w.r.t.\ its age and under neutral growth), this anomaly would suggest a positive selection of this allele. The main problem is now to be able to estimate how old the allele is. Sabeti and al. remarked that the allelic partition can be used as a clock to estimate the age of an allele. More precisely, their study begins by selecting a small region of chromosome which characterized the presence of the allele under consideration. Now, the type of an individual, at a distance $x$ (measured in kb) from the core region, is the sequence of $x$ bases following the core region (excluded). As a consequence, the allelic partition of the subpopulation carrying the allele becomes thinner as $x$ increases (because the higher $x$ is, the higher is the probability that a mutation occurred on the sequence of $x$ bases). Finally, the speed of fragmentation of the allelic partition, when $x$ increases, gives clues on the age of the allele. One of the purposes of this model is to understand how the frequency spectrum evolves under neutral evolution. In this work, we discuss some aspects of this method and give some directions in order to construct rigorous tests for the positive selection (see Section \ref{sec:num}).

The paper is organized as follows. Section \ref{sec:models} is devoted to the mathematical description of the model and to preliminary results which are used in the sequel. Section \ref{sec:results} gives the mains theoretical results of this work and, in particular, a central limit theorems which allow to study the error in our proposed approximations. Section \ref{proof:th2}, \ref{proof:th3}, \ref{proof:th4} are devoted to the proofs of Theorem \ref{thm:tclAllelic}, \ref{thm:tclexp} and \ref{thm:cltFin} respectively. Finally, in Section \ref{sec:num} we perform some numerical studies on the model to stress the quality of our approximation. The discussions about the method of Sabeti and al. are given in this last section. An appendix contains some technical proofs and a section which is a reminder of renewal theory.
\section{Model and preliminaries}
\label{sec:models}
In this work, we consider a branching population with the following dynamic: starting with a single individual (called the \emph{ancestor}) whose lifetime is distributed according to an arbitrary probability distribution $\mathbb{P}_{V}$, this \emph{ancestor} gives birth to new individuals at a Poissonian rate $b$. Each birth event giving a single new individual. From this point, each child of the ancestor lives and gives birth according to the same mechanism independently from the other individuals in the population. This formal description can be made rigorous through the definition of a probability distribution on the set of chronological trees. For the details of such construction, we refer the reader to \cite{L10}. The first quantity of interest when studying such population is the number $N_{t}$ of alive individuals in the population at a fixed time $t$ (assuming that the time $t=0$ is birth-date of the ancestor). The process $(N_{t},\ t\in\mathbb{R}_{+})$ is known as binary homogeneous Crump-Mode-Jagers process and is a simple example of non-Markovian branching process. In the sequel, we denote by $W(t)$ the expectation of $N_{t}$ conditionally on the non-extinction at time $t$. That is
\[
W(t):=\mathbb{E}\left[N_{t}\mid N_{t}>0  \right].
\] 

In \cite{L10}, the author shows that the random variable $N_{t}$ is geometrically distributed under $\mathbb{P}_{t}$ with parameter $\frac{1}{W(t)}$.
In addition, the author of \cite{L10} showed that the Laplace transform of $W$ can be linked to the Laplace transform of $\mathbb{P}_{V}$ through the relation
\[
\int_{[0,\infty)}W(s)e^{-\lambda s}\ ds=\frac{1}{\psi(\lambda)},\ \forall \lambda>\alpha,
\]
where
\begin{equation}
\label{eq:laplace}
\psi(x)=x-\int_{(0,\infty]}\left(1-e^{-rx}\right)b\mathbb{P}_{V}(dr), \ \ x\in\mathbb{R}_{+},
\end{equation}

and $\alpha$ is the largest root of $\psi$. 
In particular, the Laplace transform of $\mathbb{P}_{V}$ can be expressed in terms of $\psi$,
\begin{equation}
\label{eq:laplacepv}
\int_{\mathbb{R}_{+}}e^{-\lambda v}\mathbb{P}_{V}(dv)=1+\frac{\psi(\lambda)-\lambda}{b}.
\end{equation}

In this work, we assume that $\alpha$ is a strictly positive real number. This case is called the supercritical case and is equivalent to $b\mathbb{E}[V]>1$. In the supercritical case, the real number $\alpha$ is called the Malthusian parameter of the population because it corresponds to the mean exponential growth rate of the population.
Before gong further, let us remark that equation \eqref{eq:laplacepv} leads easily to the following identity:
\begin{equation}
\label{eq:intalpha}
\int_{\mathbb{R}_{+}}e^{-\alpha v}\mathbb{P}_{V}(dv)=1-\frac{\alpha}{b}.
\end{equation}
Many previous works demonstrate \cite{CL1,CL2,CLR} that some properties of the splitting tree were easier to study on the tree describing only the genealogical relation between the lineages of the individuals alive at time $t$. For instance, in the model with mutations, the difference between two individuals in term of type lies only on the time past since their lineages has diverged. Hence, this particular genealogical tree, known as \emph{coalescent point processes} (CPP), contains the essential information to study the allelic partition.  In order to derive the law of that genealogical tree, one needs to characterize the joint law of the \emph{times of coalescence} between pairs of individuals in the population, which are the times since their lineages have split.

In \cite{L10}, the author defines an order on the set of individuals alive at a fixed time $t$ and consider the sequence of times of coalescences $(H_{i})_{0\leq i \leq N_{t}-1}$ between two consecutive individuals (that is $H_{i}$ is the time passed since the lineage of individuals $i$ and $i+1$ have diverged) with the convention that the older lineage is the first one (i.e.\ $H_{0}=t$).
 Moreover, in \cite{L10}, the author shows that the random vector $(H_{i})_{0\leq i \leq N_{t}-1}$ can be produced from a sequence $(H_{i})_{i\geq 1}$ of i.i.d.\ random variable stopped at its first value greater than $t$ and such that
\[
\mathbb{P}\left(H_{1}>s\right)=\frac{1}{W(s)}, \quad s\in\mathbb{R}_{+}.
\]

To summarize, given the population is still alive at time $t$, one can forget about the details of the splitting tree and code the genealogy by a new object called the {\it coalescent point process} (CPP). Its law is the law of a sequence $\left(H_{i}\right)_{0\leq i\leq N_{t}-1}$, where the family $\left(H_{i}\right)_{i\geq 1}$ is i.i.d. with the same law as $H$, stopped before its first value $H_{N_{t}}$ greater than $t$, and $H_{0}$ is deterministic equal to $t$ (see Figure \ref{fig: contour}).

Although we do not use directly the CPP in this work, this object allowed us to obtain \cite{CH} formulas for the moments of the frequency spectrum which are widely used in the sequel.
\begin{rem}
\label{rem:indep}
Let $N$ be a integer valued random variable. In the sequel we said that a random vector with random size $\left(X_{i}\right)_{1\leq i\leq N}$ form an i.i.d.\ family of random variables independent of $N$, if and only if
\[
\left(X_{1},\dots,X_{N}\right)\overset{d}{=}\left(\tilde{X_{1}},\dots,\tilde{X}_{N}\right),
\] 
where $\left(\tilde{X}_{i}\right)_{i\geq 1}$ is a sequence of i.i.d.\ random variables distributed as $X_{1}$ independent of $N$.
\end{rem}

\begin{figure}[ht]

\unitlength 2mm 
\linethickness{0.4pt}

\begin{picture}(66,33)(-5,10)
\put(4,39.875){\line(1,0){62}}
\put(10,40){\line(0,-1){9}}
\put(14,40){\line(0,-1){11.5}}
\put(18,40){\line(0,-1){4}}
\put(22,40){\line(0,-1){7}}
\put(26,40){\line(0,-1){16}}
\put(30,40){\line(0,-1){6}}
\put(34,39.875){\line(0,-1){8.5}}
\put(38,40){\line(0,-1){5.5}}
\put(42,40){\line(0,-1){11.5}}
\put(46,40){\line(0,-1){3.625}}
\put(50,40){\line(0,-1){22}}
\put(54,40){\line(0,-1){5}}
\put(58,40){\line(0,-1){7.5}}
\put(62,39.875){\line(0,-1){5}}
\put(6,40){\line(0,-1){25}}
\put(6,15){\line(0,-1){.8}}
\put(6,13.33){\line(0,-1){.8}}
\put(6,11.73){\line(0,-1){.8}}
\put(9.93,30.93){\line(-1,0){.8}}
\put(8.33,30.93){\line(-1,0){.8}}
\put(6.73,30.93){\line(-1,0){.8}}
\put(13.93,28.43){\line(-1,0){.8889}}
\put(12.152,28.43){\line(-1,0){.8889}}
\put(10.374,28.43){\line(-1,0){.8889}}
\put(8.596,28.43){\line(-1,0){.8889}}
\put(6.819,28.43){\line(-1,0){.8889}}
\put(17.805,35.93){\line(-1,0){.8}}
\put(16.205,35.93){\line(-1,0){.8}}
\put(14.605,35.93){\line(-1,0){.8}}
\put(21.93,32.93){\line(-1,0){.8889}}
\put(20.152,32.93){\line(-1,0){.8889}}
\put(18.374,32.93){\line(-1,0){.8889}}
\put(16.596,32.93){\line(-1,0){.8889}}
\put(14.819,32.93){\line(-1,0){.8889}}
\put(25.93,23.93){\line(-1,0){.9524}}
\put(24.025,23.93){\line(-1,0){.9524}}
\put(22.12,23.93){\line(-1,0){.9524}}
\put(20.215,23.93){\line(-1,0){.9524}}
\put(18.311,23.93){\line(-1,0){.9524}}
\put(16.406,23.93){\line(-1,0){.9524}}
\put(14.501,23.93){\line(-1,0){.9524}}
\put(12.596,23.93){\line(-1,0){.9524}}
\put(10.692,23.93){\line(-1,0){.9524}}
\put(8.787,23.93){\line(-1,0){.9524}}
\put(6.882,23.93){\line(-1,0){.9524}}
\put(29.93,33.93){\line(-1,0){.8}}
\put(28.33,33.93){\line(-1,0){.8}}
\put(26.73,33.93){\line(-1,0){.8}}
\put(33.93,31.43){\line(-1,0){.8889}}
\put(32.152,31.43){\line(-1,0){.8889}}
\put(30.374,31.43){\line(-1,0){.8889}}
\put(28.596,31.43){\line(-1,0){.8889}}
\put(26.819,31.43){\line(-1,0){.8889}}
\put(37.93,34.43){\line(-1,0){.8}}
\put(36.33,34.43){\line(-1,0){.8}}
\put(34.73,34.43){\line(-1,0){.8}}
\put(41.93,28.43){\line(-1,0){.9412}}
\put(40.047,28.43){\line(-1,0){.9412}}
\put(38.165,28.43){\line(-1,0){.9412}}
\put(36.283,28.43){\line(-1,0){.9412}}
\put(34.4,28.43){\line(-1,0){.9412}}
\put(32.518,28.43){\line(-1,0){.9412}}
\put(30.636,28.43){\line(-1,0){.9412}}
\put(28.753,28.43){\line(-1,0){.9412}}
\put(26.871,28.43){\line(-1,0){.9412}}
\put(45.93,36.43){\line(-1,0){.8}}
\put(44.33,36.43){\line(-1,0){.8}}
\put(42.73,36.43){\line(-1,0){.8}}
\put(49.93,17.93){\line(-1,0){.9778}}
\put(47.974,17.93){\line(-1,0){.9778}}
\put(46.019,17.93){\line(-1,0){.9778}}
\put(44.063,17.93){\line(-1,0){.9778}}
\put(42.107,17.93){\line(-1,0){.9778}}
\put(40.152,17.93){\line(-1,0){.9778}}
\put(38.196,17.93){\line(-1,0){.9778}}
\put(36.241,17.93){\line(-1,0){.9778}}
\put(34.285,17.93){\line(-1,0){.9778}}
\put(32.33,17.93){\line(-1,0){.9778}}
\put(30.374,17.93){\line(-1,0){.9778}}
\put(28.419,17.93){\line(-1,0){.9778}}
\put(26.463,17.93){\line(-1,0){.9778}}
\put(24.507,17.93){\line(-1,0){.9778}}
\put(22.552,17.93){\line(-1,0){.9778}}
\put(20.596,17.93){\line(-1,0){.9778}}
\put(18.641,17.93){\line(-1,0){.9778}}
\put(16.685,17.93){\line(-1,0){.9778}}
\put(14.73,17.93){\line(-1,0){.9778}}
\put(12.774,17.93){\line(-1,0){.9778}}
\put(10.819,17.93){\line(-1,0){.9778}}
\put(8.863,17.93){\line(-1,0){.9778}}
\put(6.907,17.93){\line(-1,0){.9778}}
\put(53.93,34.93){\line(-1,0){.8}}
\put(52.33,34.93){\line(-1,0){.8}}
\put(50.73,34.93){\line(-1,0){.8}}
\put(57.93,32.43){\line(-1,0){.8889}}
\put(56.152,32.43){\line(-1,0){.8889}}
\put(54.374,32.43){\line(-1,0){.8889}}
\put(52.596,32.43){\line(-1,0){.8889}}
\put(50.819,32.43){\line(-1,0){.8889}}
\put(61.93,34.93){\line(-1,0){.8}}
\put(60.33,34.93){\line(-1,0){.8}}
\put(58.73,34.93){\line(-1,0){.8}}
\put(6,41){\makebox(0,0)[cc]{$0$}}
\put(10,41){\makebox(0,0)[cc]{$1$}}
\put(14,41){\makebox(0,0)[cc]{$2$}}
\put(18,41){\makebox(0,0)[cc]{$3$}}
\put(22,41){\makebox(0,0)[cc]{$4$}}

\put(26,41){\makebox(0,0)[cc]{$5$}}
\put(30,41){\makebox(0,0)[cc]{$6$}}
\put(34,41){\makebox(0,0)[cc]{$7$}}
\put(38,41){\makebox(0,0)[cc]{$8$}}
\put(42,41){\makebox(0,0)[cc]{$9$}}
\put(46,41){\makebox(0,0)[cc]{$10$}}

\put(50,41){\makebox(0,0)[cc]{$12$}}
\put(54,41){\makebox(0,0)[cc]{$13$}}
\put(58,41){\makebox(0,0)[cc]{$14$}}
\put(62,41){\makebox(0,0)[cc]{$15$}}
\end{picture}

\caption{ A coalescent point process for $16$ individuals, hence $15$ branches. (Image by A. Lambert)}
\label{fig: contour}
\end{figure}
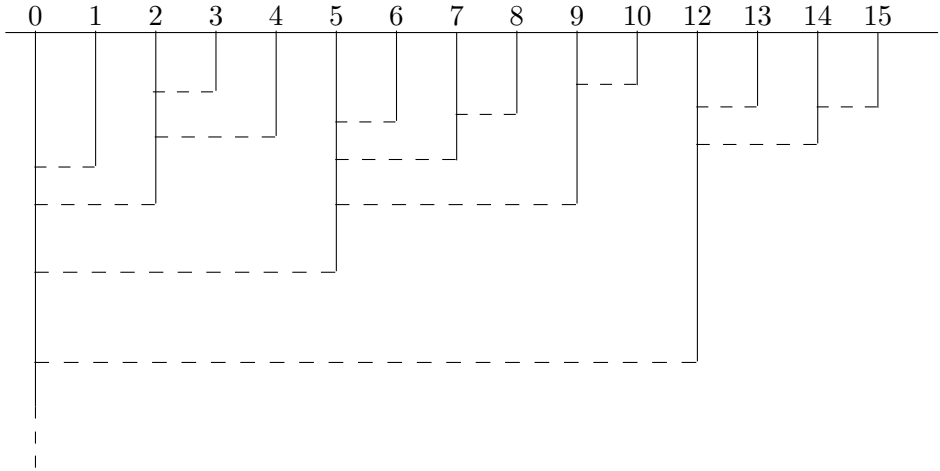

 Before going further, let us point out that if we define $N_{t}$ as the first value of the sequence $\left(H_{i}\right)_{i\geq 1}$ greater than $t$, i.e.\
 \[
 N_{t}=\inf\{i\geq 1 \mid H_{i}>t\},
 \]
 then $N_{t}$ is indeed geometric with the expected parameter. More precisely, for a positive integer $k$,
 \begin{equation}
 \label{eq:loiNt}
 \mathbb{P}\left(N_{t}=k\mid N_{t}>0\right)=\frac{1}{W(t)}\left(1-\frac{1}{W(t)}\right)^{k-1}.
 \end{equation}
 In particular,
  \begin{equation}
  \label{eq:momNt}
  \mathbb{E}\left[N_{t}\mid N_{t}>0\right]=W(t).
  \end{equation}
 Moreover, it can be showed (see \cite{Rich}), that
 \begin{equation}
 \label{eq:NtNoCond}
 \mathbb{E}N_{t}=W(t)-W\star\mathbb{P}_{V}(t),
 \end{equation}
 and
 \begin{equation}
 \label{eq:probaNoCond}
 \mathbb{P}\left(N_{t}>0 \right)=1-\frac{W\star\mathbb{P}_{V}(t)}{W(t)},
 \end{equation}
 where
 \[
 W\star\mathbb{P}_{V}(t):=\int_{[0,t]}W(t-s)\mathbb{P}_{V}(ds).
 \]

Now, let us introduce the mathematical formalism for the mutation process used in this work  (this formalism comes from \cite{CH}).
Since only the mutations occurring on the lineages of living individuals at time $t$ can be observed, it follows from standard properties on Poisson point processes, that the mutation process can be defined directly on the CPP. So, let $\mathcal{P}$ be a Poisson random measure on $[0,t]\times\mathbb{N}$ with intensity measure $\theta \lambda\otimes C$, where $C$ is the counting measure on $\mathbb{N}$, then the mutation random measure $\mathcal{N}$ on the CPP is defined by
 \[
 \mathcal{N}\left(da,di\right)=\mathds{1}_{H_{i}>t-a}\mathds{1}_{i<\mathcal{N}_{t}}\mathcal{P}\left(di,da\right),
 \]
 where an atom at $(a,i)$ means that the $i$th branch experiences a mutation at time $t-a$.
 We suppose that each individual inherits the type of its parent. This rule yields a partition of the population by types. The distribution of the sizes of the families in the population is called the frequency spectrum and is defined as the sequence $\left(A(k,t)\right)_{k\geq1}$ where $A(k,t)$ is the number of types carried by exactly $k$ individuals in the alive population at time $t$, excluding the family holding the ancestral type of the population (i.e. individuals holding the same type as the root at time $0$). This last family is called \emph{clonal}, as the ancestral type.
 
In the study of the frequency spectrum, an important role is played by the law of the clonal family. We denote by $Z_{0}(t)$ the size of this family at time $t$.

To study this family, it is easier to consider the clonal splitting tree constructed from the original splitting tree by cutting every branches beyond mutations. This clonal splitting tree is a standard splitting tree without mutations, where individuals are killed as soon as they die or experience a mutation. The new lifespan law is therefore the minimum between an exponential random variable of parameter $\theta$ and an independent copy of $V$.
 It is straightforward by simple manipulations of Laplace transforms that the Laplace exponent of the corresponding contour process is
 \[
 \psi_{\theta}(x)=x-\int_{(0,\infty]}\left(1-e^{-rx}\right)\Lambda_{\theta}(dr)=\frac{x\psi(x+\theta)}{x+\theta}.
 \]
 We denote by $W_{\theta}$ the corresponding scale function.
 This leads to,
 \[
 \mathbb{P}\left(Z_{0}(t)=k\mid Z_{0}(t)>0\right)=\frac{1}{W_{\theta}(t)}\left(1-\frac{1}{W_{\theta}(t)} \right)^{k-1}.
 \]
 When $\alpha>\theta$ (resp. $\alpha=\theta$, $\alpha<\theta$), this new tree is supercritical (resp. critical, sub-critical) and we talk about \emph{clonal supercritical case} (resp. \emph{critical}, \emph{sub-critical} case).
 
Moreover, the law of $Z_{0}$ conditionally on the event $\{N_{t}>0\}$ can be obtained, and is given by
 \begin{equation}
 \mathbb{P}\left(Z_{0}(t)=k\mid N_{t}>0\right)=\frac{e^{-\theta t}W(t)}{W_{\theta}(t)^{2}}\left(1-\frac{1}{W_{\theta}(t)} \right)^{k-1},\quad \forall k\geq 1. \label{eq:loizzero}
 \end{equation}
  For the rest of this paper, unless otherwise stated, the notation $\mathbb{P}_{t}$ refers to $\mathbb{P}\left(.\mid N_{t}>0 \right)$ whereas $\mathbb{P}_{\infty}$ refers to the probability measure conditioned on the non-extinction event (which has positive probability in the supercritical case).

 Finally, we recall the asymptotic behavior of the scale functions $W(t)$ and $W_{\theta}(t)$ which is widely used in the sequel,
 \begin{lem}(\cite[Thm. 3.21]{CL1})
 \label{lem: asyComp}
 There exist a positive constant $\gamma$ such that,
 \[
 e^{-\alpha t}\psi'(\alpha)W(t)-1=\mathcal{O}\left(e^{-\gamma t} \right).
 \]
 In the case that $\theta<\alpha$ (clonal supercritical case), 
 \[
 W_{\theta}(t)\underset{t\to\infty}{\sim}\frac{e^{\left(\alpha-\theta \right)t}}{\psi_{\theta}(\alpha-\theta)}.
 \]
 In the case that $\theta>\alpha$ (clonal sub-critical case),
 \[
 W_{\theta}(t)=\frac{\theta}{\psi(\theta)}+\mathcal{O}\left(e^{-\left(\theta-\alpha\right)t} \right).
 \]
 In the case where $\theta=\alpha$ (clonal critical case),
 \[
 W_{\theta}(t)\underset{t\to\infty}{\sim}\frac{\theta t}{\psi'(\alpha)}.
 \]
 \end{lem}
 For a purpose, a more precise description of the asymptotic behavior of $W$ is needed. It is given by the following result.
 \begin{lem}{\cite[Prop. 5.1]{H1}}
 	\label{lem:WComp}
There exists a positive non-increasing c\`adl\`ag function $F$ such that
\[
W(t)=\frac{e^{\alpha t}}{\psi'(\alpha)}-e^{\alpha t}F(t),\quad t\geq 0,
\]
and
\[
\lim\limits_{t\to \infty}e^{\alpha t}F(t)=
\begin{cases}
\frac{1}{b\mathbb{E}V-1} & \mbox{if}\ \mathbb{E}V<\infty, \\ 
0 & \mbox{otherwise.}
\end{cases}
\]
 \end{lem}
 From this Lemma and \eqref{eq:probaNoCond}, one can easily deduce that
 \begin{equation}
 \label{eq:Prnonex}
 \mathbb{P}\left(\text{NonEx} \right)=\lim\limits_{t\to\infty}\mathbb{P}\left(N_{t}>0 \right)=\frac{\alpha}{b},
 \end{equation}
 where $\text{NonEx}$ refer to the non-extinction event.

In \cite{CH}, we show that a CPP stopped at time $t$ with scale function $W$ can be constructed by grafting independent CPP stopped at a fixed time $a\leq t$ on a CPP stopped at time $t-a$ with an explicit scale function different of $W$ (see Figure \ref{fig : coalpointproc}).
 \begin{figure}[ht]
 \unitlength 1.6mm 
 \linethickness{0.9pt}
 \begin{picture}(100,30)(0,0)
 \put(1,-1){\makebox{\small{t}}}
 \put(-1.5,9.5){\makebox{\small{t-a}}}
 \put(3,10){\line(1,0){1}}
 \put(3.2,0){\line(1,0){1}}
 \put(3,0){
 \color{black}
 \put(0,10){\line(1,0){100}}
 \put(0,10){\line(0,-1){10}}
 \put(50,10){\line(0,-1){7}}
 \put(25,10){\line(0,-1){3}}
 \put(75,10){\line(0,-1){5}}
 \multiput(25,7)(-1,0){25}{\line(-1,0){0.5}}
 \multiput(50,3)(-1,0){50}{\line(-1,0){0.5}}
 \multiput(75,5)(-1,0){25}{\line(-1,0){0.5}}
 \setlength{\unitlength}{0.8mm}
 \linethickness{0.62pt}
 \put(0,20){
 \put(13,22){\makebox{$\mathcal{P}^{(1)}$}}
 \color{gray4}
 \put(0,0){\line(0,1){20}}
 \color{gray4}
 \put(0,20){\line(1,0){10}}
 \color{gray1}
 \put(10,20){\line(1,0){10}}
 \color{gray2}
 \put(20,20){\line(1,0){10}}
 \color{gray3}
 \multiput(30,20)(1,0){10}{\line(-1,0){0.5}}
 \color{black}
 \color{gray1}
 \put(20,20){\line(0,-1){5}}
 \color{gray1}
 \put(10,20){\line(0,-1){10}}
 \color{gray3}
 \put(30,20){\line(0,-1){15}}
 \multiput(30,5)(-1,0){30}{\line(-1,0){0.5}}
 \multiput(20,15)(-1,0){10}{\line(-1,0){0.5}}
 \multiput(10,10)(-1,0){10}{\line(-1,0){0.5}}
 }
 \put(50,20){
 \put(13,22){\makebox{$\mathcal{P}^{(2)}$}}
 \color{gray4}
 
 \put(0,0){\line(0,1){20}}
 \color{gray4}
 \put(0,20){\line(1,0){10}}
 \color{gray1}
 \put(10,20){\line(1,0){10}}
 \color{gray2}
 \put(20,20){\line(1,0){10}}
 \color{gray3}
 \multiput(30,20)(1,0){10}{\line(-1,0){0.5}}
 \color{black}
 \color{gray1}
 \put(20,20){\line(0,-1){5}}
 \color{gray1}
 \put(10,20){\line(0,-1){10}}
 \color{gray3}
 \put(30,20){\line(0,-1){15}}
 \multiput(30,5)(-1,0){30}{\line(-1,0){0.5}}
 \multiput(20,15)(-1,0){10}{\line(-1,0){0.5}}
 \multiput(10,10)(-1,0){10}{\line(-1,0){0.5}}}
 \put(100,20){\color{gray4}
 \put(0,0){\line(0,1){20}}
 \put(13,22){\makebox{$\mathcal{P}^{(3)}$}}
 \color{gray4}
 \put(0,20){\line(1,0){10}}
 \color{gray1}
 \put(10,20){\line(1,0){10}}
 \color{gray2}
 \put(20,20){\line(1,0){10}}
 \color{gray3}
 \multiput(30,20)(1,0){10}{\line(-1,0){0.5}}
 \color{black}
 \color{gray1}
 \put(20,20){\line(0,-1){5}}
 \color{gray1}
 \put(10,20){\line(0,-1){10}}
 \color{gray3}
 \put(30,20){\line(0,-1){15}}
 \multiput(30,5)(-1,0){30}{\line(-1,0){0.5}}
 \multiput(20,15)(-1,0){10}{\line(-1,0){0.5}}
 \multiput(10,10)(-1,0){10}{\line(-1,0){0.5}}}
 \put(150,20){\color{gray4}
 \put(0,0){\line(0,1){20}}
 \put(13,22){\makebox{$\mathcal{P}^{(4)}$}}
 \color{gray4}
 \put(0,20){\line(1,0){10}}
 \color{gray1}
 \put(10,20){\line(1,0){10}}
 \color{gray2}
 \put(20,20){\line(1,0){10}}
 \color{gray3}
 \multiput(30,20)(1,0){10}{\line(-1,0){0.5}}
 \color{black}
 \color{gray1}
 \put(20,20){\line(0,-1){5}}
 \color{gray1}
 \put(10,20){\line(0,-1){10}}
 \color{gray3}
 \put(30,20){\line(0,-1){15}}
 \multiput(30,5)(-1,0){30}{\line(-1,0){0.5}}
 \multiput(20,15)(-1,0){10}{\line(-1,0){0.5}}
 \multiput(10,10)(-1,0){10}{\line(-1,0){0.5}}}
 }
 \end{picture}
 \caption{
 Adjunction of trees.}
 \label{fig : coalpointproc}
 \end{figure}
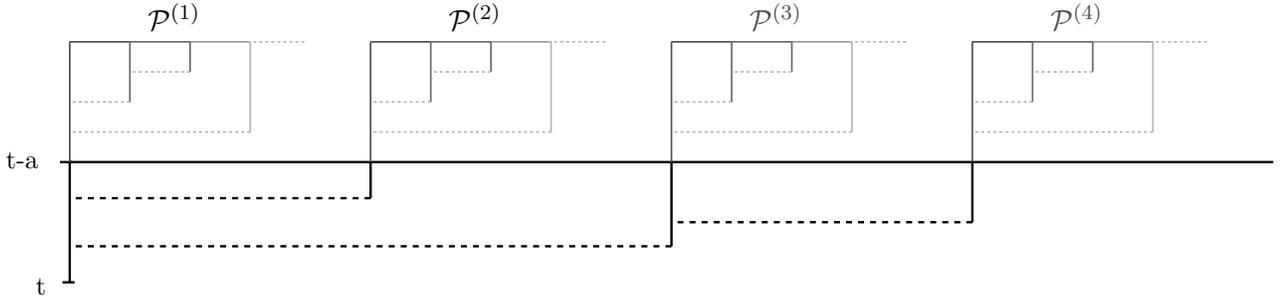
 Moreover, we showed that the frequency spectrum can be expressed as an integral with respect to the random measure $\mathcal{N}$ along the CPP, that is
 \begin{equation}
 \label{eq:recFormula}
 \prod_{i=1}^{l}A(k_{i},t)=\sum_{i=1}^{l}\int_{[0,t]\times \mathbb{N}}\mathds{1}_{Z^{(u)}_{0}(a)=k_{i}}\sum_{u_{1:l-1}=1}^{N^{(t)}_{t-a}}\prod_{\underset{i\neq j}{j=1}}^{l-1}A^{(u_{j})}(k_{j},a)\ \mathcal{N}\left(da,du \right),
 \end{equation}
 where $A^{(u)}(k,a)$ (resp. $Z^{(u)}_{0}$) refers to the frequency spectrum (resp. clonal family) of the $u$th grafted sub-CPP, and $\sum_{u_{1:l-1}=1}^{N^{(t)}_{t-a}}$ denotes for the multi-sum
 
 \[
 \sum_{u_{1}=1}^{N^{(t)}_{t-a}}\dots \sum_{u_{l-1}=1}^{N^{(t)}_{t-a}}.
 \]
 Moreover, in \cite[Thm, 3.1]{CH} we show that the expectation of such integral can be computed easily when the integrand presents local independence properties with the random measure as in formula \eqref{eq:recFormula}.
  Equation \eqref{eq:recFormula} is used later to obtain some moments estimates useful to prove our theorems.
 In particular, this allows to prove that (see \cite{CH}) for any positive integer $k$ and $l$,
 \begin{equation}
 \label{eq:Mom1akate}
 \mathbb{E}_{t}A(k,t)=W(t)\int_{0}^{t}\frac{\theta e^{-\theta s}}{W_{\theta}(s)^{2}}\left(1-\frac{1}{W_{\theta}(s)} \right)^{k-1}ds,
 \end{equation}
 and
 \begin{align}
 \label{eq:Mom2akate}
 \mathbb{E}A(k,t)A(l,t)&=2W(t)^{2}\int_{0}^{t}\frac{\theta e^{-\theta s}}{W_{\theta}(s)^{2}}\left(1-\frac{1}{W_{\theta}(s)} \right)^{k-1}ds\int_{0}^{t}\frac{\theta e^{-\theta s}}{W_{\theta}(s)^{2}}\left(1-\frac{1}{W_{\theta}(s)} \right)^{l-1}ds\notag\\
 &-W(t)\int_{0}^{t}2\theta \ \frac{e^{-\theta a}W(a)}{W_{\theta}(a)^{2}}\left(1-\frac{1}{W_{\theta}(a)} \right)^{l-1}\int_{0}^{s}\frac{\theta e^{-\theta s}}{W_{\theta}(s)^{2}}\left(1-\frac{1}{W_{\theta}(a)} \right)^{k-1}dsda\notag\\
 &-W(t)\int_{0}^{t}2\theta \ \frac{e^{-\theta a}W(a)}{W_{\theta}(a)^{2}}\left(1-\frac{1}{W_{\theta}(a)} \right)^{k-1}\int_{0}^{s}\frac{\theta e^{-\theta s}}{W_{\theta}(s)^{2}}\left(1-\frac{1}{W_{\theta}(a)} \right)^{l-1}dsda\notag\\
 &+W(t)\mathbb{E}\int_{0}^{t}\theta W(a)^{-1}\left(\mathbb{E}\left[A(k,t)\mathds{1}_{Z_{0}(a)=l} \right]+\mathbb{E}\left[A(l,t)\mathds{1}_{Z_{0}(a)=k} \right]\right)da\notag\\
 &+\mathds{1}_{l=k}W(t)\int_{0}^{t}\frac{\theta e^{-\theta s}}{W_{\theta}(s)^{2}}\left(1-\frac{1}{W_{\theta}(s)} \right)^{k-1}ds.
 \end{align}

 These tools also allow, for instance, to prove next two results \cite{L10,CLR,CH}.
 \begin{thm}
 \label{thm:ASCVNt}
 There exists a random variable $\mathcal{E}$, such that
 \[
 \lim\limits_{t\to\infty}e^{-\alpha t}N_{t}=\frac{\mathcal{E}}{\psi'(\alpha)},\quad a.s. \text{ and in } L^{2}.
 \]
 Moreover, under $\mathbb{P}_{\infty}$, $\mathcal{E}$ is exponentially distributed with parameter one.
 \end{thm}
 \begin{thm}
 \label{thm:ASCVakate}
 For any positive integer $k$,
 \[
 \lim\limits_{t\to\infty}e^{-\alpha t}A(k,t)=\frac{c_{k}\mathcal{E}}{\psi'(\alpha)},\quad a.s. \text{ and in } L^{2},
 \]
 where $\mathcal{E}$ is the random variable of the Theorem \ref{thm:ASCVNt} and 
 \begin{equation}
 \label{eq:ck}
 c_{k}=\int_{0}^{\infty}\frac{\theta e^{-\theta a}}{W_{\theta}(a)}\left( 1-\frac{1}{W_{\theta}(a)}\right)^{k-1}da.
 \end{equation}
 \end{thm}

\section{Main results}
\label{sec:results}
 
The a.s.\ convergence stated in Section \ref{sec:models} suggests studying the second order properties of the convergence to get central limit theorem. Our main result, Theorem \ref{thm:cltFin}, allows to study the asymptotic error in the approximation $2$ proposed in the introduction of this work. In addition, we prove more standard central limit theorems which are interesting from the theoretical point of view.

Before going further, we recall that the Laplace distribution with mean $\mu\in\mathbb{R}^{n}$ and covariance matrix $K$ is the probability distribution whose characteristic function is given, for all $\lambda\in\mathbb{R}^{n}$ by
\[
\frac{1}{1+\frac{1}{2}\lambda'K\lambda-i\mu'\lambda}
\]
We denote this law by $\mathcal{L}\left(\mu,K \right)$. We also recall that, if $G$ is a Gaussian random vector with mean $\mu$ and covariance matrix $K$ and $\mathcal{E}$ is an exponential random variable with parameter $1$ independent of $G$, then $\sqrt{\mathcal{E}}G$ is Laplace $\mathcal{L}\left(\mu,K \right)$.

\subsection{CLT for the convergence of Theorem \ref{thm:ASCVakate}}

\begin{thm}
\label{thm:tclAllelic}
Suppose that $\theta >\alpha$ and $\int_{[0,\infty)}e^{\left(\theta-\alpha\right)v}\mathbb{P}_{V}(dv)>1$ . Then, we have, under $\mathbb{P}_{\infty}$,
\[
\left(e^{\alpha\frac{t}{2}}\left(\psi'(\alpha)A(k,t)-e^{\alpha t}c_{k}\mathcal{E}\right)\right)_{k\in \mathbb{N}}\xrightarrow[t\to\infty]{(d)}\mathcal{L}\left(0,K\right),
\]
where $K$ is some covariance matrix and the constants $c_{k}$ are defined in \eqref{eq:ck}.
\end{thm}
The proof of this result can be found in Section \ref{proof:th2}.
\begin{rem}
We are not able to compute explicitly the covariance matrix $K$ in the general case due to our method of demonstration. However, all our other results give explicit formulas. In particular, the case where $\mathbb{P}_{V}$ is exponential is given by the next theorem. The Yule case is also covered in the following theorem for $d=0$ although it does not satisfy the hypothesis of Theorem \ref{thm:tclAllelic}.
\end{rem}

\begin{thm}
\label{thm:tclexp}
Suppose that $V$ is exponentially distributed with parameter $d\in[0,b)$. In this case, $\alpha=b-d$. We still suppose that $\alpha<\theta$, then
\[
\left(e^{\alpha\frac{t}{2}}\left(\psi'(\alpha)A(k,t)-e^{\alpha t}c_{k}\mathcal{E}\right)\right)_{k\in N}\xrightarrow[t\to\infty]{(d)}\mathcal{L}\left(0,K\right), \ \text{w.r.t.}\ \mathbb{P}_{\infty},
\]
where $K$ is given by
\[
K_{l,k}=M_{l,k}+c_{k}c_{l}\frac{\alpha}{b}\left(1-6\frac{d}{\alpha}\right),
\]
and
\begin{align}
\label{eq:Cov}
&M_{l,k}=\notag\\&2\psi'(\alpha)\int_{0}^{\infty}\frac{\theta e^{-\theta a}}{W_{\theta}(a)^{2}}\left(\left(1-\frac{1}{W_{\theta}(a)}\right)^{l-1}\left(\mathbb{E}_{a}\left[A(k,a)\right]-c_{k}W(a)\right)+\left(1-\frac{1}{W_{\theta}(a)}\right)^{k-1}\left(\mathbb{E}_{a}\left[A(l,a)\right]-c_{l}W(a)\right)\right)da\notag\\
&-\psi'(\alpha)\int_{0}^{\infty}\theta W(a)^{-1}\mathbb{E}_{a}\left[\left(A(k,a)-c_{k}N_{a}\right)\mathds{1}_{Z_{0}(a)=l}+\left(A(l,a)-c_{l}N_{a}\right)\mathds{1}_{Z_{0}(a)=k} \right]\notag\\&+\mathds{1}_{l=k}\int_{0}^{\infty}\frac{\theta e^{-\theta s}}{W_{\theta}(s)^{2}}\left(1-\frac{1}{W_{\theta}(s)} \right)^{k-1}ds,
\end{align}
where $W$, $W_{\theta}$, $\psi'(\alpha)$ are defined in the Section \ref{sec:models}.
\end{thm}
The proof of this result can be found in Section \ref{proof:th4}.
Note that an explicit formula for $\mathbb{E}_{t}A(k,t)$ is given by \eqref{eq:Mom1akate}. Explicit formulas for $\mathbb{E}_{t}\left[A(k,t)\mathds{1}_{Z_{0}(t)=l} \right]$ can also be found in Proposition 4.5 of \cite{CH}, and a formula for $\mathbb{E}_{t}\left[N_{a}\mathds{1}_{Z_{0}(t)=k} \right]$ can be found in Proposition 4.1 of \cite{CL1}.

\begin{rem}
The condition on $V$ in Theorem \ref{thm:tclAllelic} is required only to ensure controls of the moments of the considered quantities. However, although the Yule case does not satisfy this condition ($V=\infty$ p.s.) it is included in this last theorem (d=0). This suggests that the condition on $V$ may not be needed.
\end{rem}
\subsection{CLT for the error between $A(k,t)$ and $c_{k}N_{t}$}
The next theorem concerns the error between $A(k,t)$ and $c_{k}N_{t}$. Once again, we have an explicit expression of the covariance matrix of the limit.
\begin{thm}
\label{thm:cltFin}
Suppose that $\theta>\alpha$, then
\[
\psi'(\alpha)\left(e^{\alpha\frac{t}{2}}\left(A(k,t)-c_{k}N_{t}\right)\right)_{k\in \mathbb{N}}\xrightarrow[t\to\infty]{(d)}\mathcal{L}\left(0, M\right), \ \text{w.r.t.}\ \mathbb{P}_{\infty},
\]
where $M$ is defined in relation \eqref{eq:Cov}.
\end{thm}
The proof of this result can be found in Section \ref{proof:th3}.
\begin{rem}
We do not known yet if the exponential random variable appearing the Gaussian mixing leading to a Laplace distribution is the same as the exponential limit of $e^{-\alpha t}A(k,t)$. However, the CLT for Markov branching processes in \cite{AN} suggest that it is, actually, the case. If, this is true in our case, it would be enough to know the correlations between the limits involved in Theorem \ref{thm:tclAllelic} and \ref{thm:cltFin} to obtain an explicit expression for the covariance matrix in Theorem \ref{thm:tclAllelic}.
\end{rem}
\section{Proof of Theorem \ref{thm:tclAllelic}}
\label{proof:th2}
The proof of this theorem is based on the proof of the central limit theorem for the process $(N_{t},\ t\in\mathbb{R}_{+})$ given in \cite{H1}. The structure of the proof follows the same lines and is detailed in Section 4 of \cite{H1}. In a sake of conciseness, we only highlight the difficulties arising in our new context. The results which are straightforward rewording of the proofs given in \cite{H1} are left to the reader.  However, we think it is necessary to recall some aspects of \cite{H1}, in particular from \cite[Section 4]{H1}. First, we recall that there exists a family $(N^{(i)}_{t},\  t\in\mathbb{R}_{+} )_{i\geq 1}$ of i.i.d.\ population counting processes with the same law as $\left(N_{t},\ t\in\mathbb{R}_{+}\right)$, and a Poisson random measure $\xi$ on $\mathbb{R}_{+}$ with intensity $b\, da$ such that
\begin{equation}
\label{eq:decomposition}
N_{t}=\int_{[0,t]}N^{(\xi_{u})}_{t-u}\mathds{1}_{V_{\emptyset}>u}\ \xi(du)+\mathds{1}_{V_{\emptyset}> t}, \quad \text{almost surely},
\end{equation}where $\xi_{u}=\xi\left([0,u]\right)$. In addition, we have that 
 $t\to\mathbb{E}\left[N_{t}\mathcal{E}\right]$ is the unique solution bounded on finite intervals of the renewal equation,
 \begin{align}
 f(t)=&\int_{\mathbb{R}_{+}}f(t-u)be^{-\alpha u}\mathbb{P}\left(V>u\right)du\nonumber
 \\&+\alpha b\mathbb{E}\left[N_{\cdot}\right]\star \left( \int_{\mathbb{R}_{+}}e^{-\alpha v}\mathbb{P}\left(V>\cdot,V>v\right)dv\right)(t)\nonumber \\&+\alpha\int_{\mathbb{R}_{+}}e^{-\alpha v}\mathbb{P}\left(V>t,V>v\right)dv\label{eq:renewNt},
 \end{align}
 and it is given by
 \begin{equation}
 \label{eq:eqeq}
 \mathbb{E}\left[N_{t}\mathcal{E}\right]=\left(1+\frac{\alpha}{b}-e^{-\alpha t}\right)W(t)-\left(1-e^{-\alpha t} \right)W\star\mathbb{P}_{V}(t).
 \end{equation}
 We also recall that equation \eqref{eq:renewNt} is obtained by taking the product $N_{t}N_{s}$, for some real number $t$ and $s$. Now, equation \eqref{eq:decomposition} allows to obtain a renewal equation for $\mathbb{E}[N_{t}N_{s}]$ which leads to \eqref{eq:renewNt} when taking the limit in $s$ of the renormalized equation.
 We also recall that Lemma \ref{lem:WComp} and equation \eqref{eq:probaNoCond} gives
 
 	\begin{equation}
 	\label{eq:survie}
 	\frac{1}{\mathbb{P}\left(N_{t}>0\right)}=\frac{b}{\alpha}-\frac{b\mu\psi'(\alpha)}{\alpha}e^{-\alpha t}+o(e^{-\alpha t}).
 	\end{equation}
This also leads, in conjunction with equation \eqref{eq:eqeq}, to
 	\begin{equation}
 	\label{eq:AsNtE}
 	\mathbb{E}_{t}N_{t}\mathcal{E}=\frac{2e^{\alpha t}}{\psi'(\alpha)}-\frac{1}{\psi'(\alpha)}-3\mu+o(1).
 	\end{equation}

 \medskip
 Finally, let us recall that for any fixed time $u$, there is a natural order (for instance given by the contour process \cite{L10}) of the individuals alive at this time. Moreover, we denote, for $1\leq i\leq N_{t}$, $O_{i}^{(u)}$ the residual lifetime of the $i$th individual alive at time $u$. The law of the vector $(O^{(u)}_{2},\dots,O^{(u)}_{N_{u}})$ is given by the following lemma which comes from \cite{H1}.
 \begin{lem}
 	\label{lem:residual}
 	Let $u$ in $\mathbb{R}_{+}$, we denote by $O_{i}$ for $i$ an integer between $1$ and $N_{u}$ the residual lifetime of the $i$th individuals alive at time $u$. Then under $\mathbb{P}_{u}$, the family $\left(O_{i},\ i\in\{1,\dots, N_{u}\}\right)$ form a family of independent random variables, independent of $N_{u}$, and, expect $O_{1}$, having the same distribution, given by, for $2\leq i\leq N_{t}$,
 	\begin{equation}
 	\label{eq:loiOversh}
 	\mathbb{P}_{u}(O_{i}\in dx)=\int_{\mathbb{R}_{+}}\ \frac{W(u-y)}{W(u)-1}b \mathbb{P}\left(V-y\in dx \right)\ dy.
 	\end{equation}
 	Moreover, it follows that the family $\left(N_{s}(O_{i}),s\in\mathbb{R}_{+}\right)_{1\leq i\leq N_{u}}$ is an independent family of process, i.i.d.\ for $i\geq 2$, and independent of $N_{u}$.
 \end{lem}
 
 \medskip
 To end this reminder, let us recall the decomposition of the limiting random variable $\mathcal{E}$ (given for instance in Theorem \ref{thm:ASCVNt}) at a fixed time $u$.
 \begin{lem}{\cite[Lemma 6.8]{H1}}
 	\label{lem:dec}
 	We have the following decomposition of $\mathcal{E}$,
 	\[
 	\mathcal{E}=e^{-\alpha u}\sum_{i=1}^{N_{u}}\mathcal{E}_{i}\left(O_{i}\right), \quad a.s.
 	\]
 	Moreover, under $\mathbb{P}_{u}$, the random variables $\left(\mathcal{E}_{i}\left(O_{i}\right)\right)_{i\geq 1}$ are independent, independent of $N_{u}$, and identically distributed for $i\geq 2$.
 \end{lem}
 We can now start the proof of theorem \ref{thm:tclAllelic}.
As in \cite{H1}, the proof begins by some estimate on moments.
\subsection{Preliminary moments estimates}
We start by computing the moment in the case of a standard splitting tree. According to \cite[Section 4]{H1}, the next step is to obtain the same kind of estimates in the case of a splitting tree whose ancestor individual has a lifetime distribution which can be different from the rest of the population.
\subsubsection{Case $V_{\emptyset}\overset{\mathcal{L}}{=}V$}
One of the main difficulties to extend the preceding proof to the frequency spectrum is to get estimates on
\[
\mathbb{E}\left[\left(\psi'(\alpha)A(k,t)-e^{\alpha t}c_{k}\mathcal{E} \right)^{n}\right], \ \text{for } n=2\text{ or }3.
\]
We first study the renewal equation satisfied by $\mathbb{E}A(k,t)\mathcal{E}$ similarly as in \cite[Lemma 6.1]{H1}.
\begin{lem}[Joint moment of $\mathcal{E}$ and $A(k,t)$]
\label{lem:jointMomAk}
$\mathbb{E}\left[A(k,t)\mathcal{E}\right]$ is the unique solution bounded on finite intervals of the renewal equation,
\begin{align}
\label{eq:renewAk}
f(t)=&\int_{\mathbb{R}_{+}}f(t-u)be^{-\alpha u}\mathbb{P}\left(V>u\right)du\notag\\&+\alpha\mathbb{E}\left[A(k,.)\right]\star b\left( \int_{\mathbb{R}_{+}}e^{-\alpha v}\mathbb{P}\left(V>.,V>v\right)dv\right)(t)\notag\\&+\alpha\mathbb{E}\left[\mathcal{E}X_{t}\right],
\end{align}
with $X_{t}$ the number of families of size $k$ alive at time $t$ whose original mutation has taken place during the lifetime of the ancestor individual.
\end{lem}
\begin{proof}
We recall that $A(k,t)$ is the number of non-ancestral families of size $k$ at time $t$. Similarly, as for $N_{t}$, $A(k,t)$ can be obtained as the sum of the contributions of all the trees grafted on the lifetime of the ancestor individual in addition to the mutations which take place on the ancestral branch, that is,
\[
A(k,t)=\int_{[0,t]}A(k,t-u,\xi_{u})\mathds{1}_{V_{\emptyset}>u}\xi(du)+X_{t},
\]  
where $\left(A(k,t,i), t\in\mathbb{R}_{+}\right)_{i\geq 1}$ is a family of independent processes having the same law as $A(k,t)$.
Now, taking the product $A(k,t)N_{s}$ and using the same arguments as in the proof of lemma \cite[Lemma 6.1]{H1} to take the limit in $s$ leads to the result.
In particular, the last term is obtained using that
\[
\lim\limits_{s\to\infty}\mathbb{E}\left[X_{t}\frac{N_{s}}{W(s)}\right]=\mathbb{E}\left[X_{t}\mathcal{E}\right].
\]
\end{proof}

The result of Lemma \ref{lem:jointMomAk} is quite disappointing since the presence of the mysterious process $X_{t}$ prevents any explicit resolution of equation \eqref{eq:renewAk}.
However, one may note that equation \eqref{eq:renewAk} is quite similar to equation \eqref{eq:renewNt} driving $\mathbb{E}N_{t}\mathcal{E}$, so if the contribution of $X_{t}$ in the renewal structure of the process is small enough, one can expect the same asymptotic behavior for $\mathbb{E}A(k,t)\mathcal{E}$ as for $\mathbb{E}N_{t}\mathcal{E}$. Moreover, we clearly have on $X_{t}$ the following a.s. estimate,
\begin{equation}
\label{eq:majXt}
X_{t}\leq \int_{[0,t]}\mathds{1}_{Z^{(u)}_{0}(t-u)>0}\mathds{1}_{V>u}\xi(du),
\end{equation}
where $Z^{(i)}_{0}$ denote for the ancestral families on the $i$th trees grafted on the ancestral branch. Hence, if we take $\theta>\alpha$ and we suppose $V<\infty$ a.s., one can expect that $X_{t}$ decreases very fast. These are the ideas the following Lemma is based on. Moreover, as it is seen in the proof of the following lemma, the hypothesis $V<\infty$ a.s. can be weakened.
\begin{lem}
\label{lem:akegnt}
Under the hypothesis of Theorem \ref{thm:tclAllelic}, for all $k\geq1$, there exists a constant $\gamma_{k}\in\mathbb{R}$ such that,
\begin{equation}
\lim\limits_{t\to\infty}\mathbb{E}N_{t}\mathcal{E}c_{k}-\mathbb{E}A(k,t)\mathcal{E}=\gamma_{k}.
\end{equation}
\end{lem}
\begin{proof}

Combining equations $\eqref{eq:renewNt}$ and $\eqref{eq:renewAk}$, we get that,
\begin{align*}
&\mathbb{E}N_{t}\mathcal{E}c_{k}-\mathbb{E}A(k,t)\mathcal{E}=\int_{\mathbb{R}_{+}}\left(\mathbb{E}N_{t-u}\mathcal{E}c_{k}-\mathbb{E}A(k,t-u)\mathcal{E} \right)be^{-\alpha u}\mathbb{P}\left(V>u\right)du\\&\quad+\underbrace{\alpha b\left(c_{k}\mathbb{E}N_{.}-\mathbb{E}\left[A(k,.)\right]\right)\star \left( \int_{\mathbb{R}_{+}}e^{-\alpha v}\mathbb{P}\left(V>.,V>v\right)dv\right)(t)}_{:=\xi^{(k)}_{1}(t)}\\&\quad\quad +\underbrace{c_{k}\mathbb{P}\left(V>t\right)-\alpha\mathbb{E}\left[X_{t}\mathcal{E}\right]}_{:=\xi^{(k)}_{2}(t)},
\end{align*}
which is also a renewal equation.
On one hand, using equations \eqref{eq:momNt} and \eqref{eq:Mom1akate} imply that
\[
\mathbb{E}_{t}\left[c_{k}N_{t}-A(k,t)\right]=W(t)\int_{t}^{\infty}\frac{\theta e^{-\theta s}}{W_{\theta}(s)^{2}}\left(1-\frac{1}{W_{\theta}(s)} \right)^{k-1}\ ds,
\]
which leads using Lemma \ref{lem: asyComp}, to
\begin{align}
\label{eq:inequality}
\xi_{1}(t)=&\alpha\int_{\mathbb{R}_{+}}\left(c_{k}\mathbb{E}N_{t-u}-\mathbb{E}\left[A(k,t-u)\right] \right)\int_{\mathbb{R}_{+}}e^{-\alpha v}\mathbb{P}\left(V>u,V>v \right)dvdu\notag\\
&\leq\mathcal{C}\int_{[0,t]}e^{(\alpha-\theta)t-u}\mathbb{P}\left(V>u\right)du\int_{[0,\infty)}e^{-\alpha u}du\notag\\
&\leq \frac{\mathcal{C}}{\alpha}e^{-\left(\theta-\alpha\right)t}\int_{0}^{t}e^{(\theta-\alpha )u}\mathbb{P}\left(V>u\right)du,
\end{align}
for some positive real constant $\mathcal{C}$.

The derivative of the r.h.s. of \eqref{eq:inequality} is given by
\begin{equation}
\label{eq:derivative}
\frac{\mathcal{C}}{\alpha}e^{-(\theta-\alpha) t}\left(e^{(\theta-\alpha) t}\mathbb{P}\left(V>t  \right)-(\alpha-\theta) \int_{0}^{t}e^{(\theta-\alpha )u}\mathbb{P}\left(V>u\right)du\right), \quad  t>0,
\end{equation}
which is equal to
\[
\frac{\mathcal{C}}{\alpha}e^{-(\theta-\alpha) t}\left(1-\int_{[0,t]}e^{(\theta-\alpha) s}\mathbb{P}_{V}(ds)\right), \quad t>0,
\]
using Stieljes integration by parts. Now, since,
\[
\int_{[0,\infty)}e^{\left(\theta-\alpha\right)s}\mathbb{P}_{V}(ds)>1,
\]
 this shows that the right hand side of \eqref{eq:inequality} is decreasing for $t$ large enough. Moreover, it is straightforward to shows that the r.h.s.\ of \eqref{eq:inequality} is also integrable.
 This implies that $\xi_{1}^{(k)}$ is DRI from the same Lemma.
On the other hand, it follows from \eqref{eq:majXt} that
\begin{equation}
X_{t}\mathcal{E}\leq \mathcal{E}\int_{[0,t]}\mathds{1}_{Z^{(u)}_{0}(t-u)>0}\mathds{1}_{V>t}\xi(du).
\end{equation}

Then, we obtain using Cauchy-Schwarz inequality, that
\[
\mathbb{E}\left[X_{t}\mathcal{E}\right]\leq \sqrt{\frac{2\alpha}{b}}\mathbb{E}\left[\left(\int_{[0,t]}\mathds{1}_{Z^{(u)}_{0}(t-u)>0}\mathds{1}_{V>t}\xi(du) \right)^{2}\right]^{1/2}.
\]
It follows that we need to investigate the behavior of
\begin{equation*}
\mathbb{E}\left[\left(\int_{(0,t)}\mathds{1}_{Z^{(u)}_{0}(t-u)>0}\mathds{1}_{V>t}\xi(du) \right)^{2}\right],
\end{equation*}
which is equal to
\begin{equation*}
\int_{0}^{t}\mathbb{P}\left(Z_{0}(t-u)>0 \right)\mathbb{P}\left(V>t\right) bdu+ \int_{[0,t]^{2}}\mathbb{P}\left(Z_{0}(t-v)>0 \right)\mathbb{P}\left(Z_{0}(t-u)>0 \right)\mathbb{P}\left(V>u,V>v\right) b^{2}du\ dv,
\end{equation*}
using \cite[Lemma 2.6]{H1}.
Then, since, from \eqref{eq:loizzero} and Lemma \ref{lem: asyComp},
\[
\mathbb{P}_{t-u}\left(Z_{0}(t-u)>0\right)=\frac{e^{-\theta (t-u)}W(t-u)}{W_{\theta}(t-u)}=\mathcal{O}(e^{-(\theta-\alpha)(t-u)}),
\]
it follows, using that the right hand side of \eqref{eq:inequality} is DRI and Lemma \ref{lem:DRI}, that $\xi^{(k)}_{2}$ is DRI.
Finally, it comes from Theorem \ref{thm:nonLatticeRen}, that
\begin{equation}
\label{eq:AsNtAkate}
\lim\limits_{t\to\infty}\mathbb{E}N_{t}\mathcal{E}c_{k}-\mathbb{E}A(k,t)\mathcal{E}=\frac{\alpha}{\psi'(\alpha)}\int_{\mathbb{R}_{+}}\xi_{1}^{(k)}(s)+\xi^{(k)}_{2}(s)ds.
\end{equation}
\end{proof}
Using the preceding lemma, we can now get the quadratic error in the convergence of the frequency spectrum.
\begin{lem}[Quadratic error for the convergence of $A(k,t)$.]
\label{lem:quadConv}
Let $k$ and $l$ two positive integers.
Then under the hypothesis of Theorem \ref{thm:tclAllelic}, there exists a family of real numbers $\left(a_{k,l}\right)_{l,k\geq 1}$ such that,  
\[
\lim\limits_{t\to\infty}e^{-\alpha t}\mathbb{E}\left[\left(\psi'(\alpha)A(k,t)-e^{\alpha t}\mathcal{E}c_{k}\right)\left(\psi'(\alpha)A(l,t)-e^{\alpha t}\mathcal{E}c_{l}\right)\right]=\frac{\alpha}{b}a_{k,l},
\]
where the sequence $\left(c_{k} \right)_{k\geq 1}$ is defined by \eqref{eq:ck}.
\end{lem}
\begin{proof}

Now, noting
\begin{equation}
\label{eq:ckt}
c_{k}(t):=\int_{0}^{t}\frac{\theta e^{-\theta a}}{W_{\theta}(a)^{2}}\left(1-\frac{1}{W_{\theta}(a)} \right)^{k-1}da,
\end{equation}
we have, from  \eqref{eq:Mom2akate} and Lemma \ref{lem:WComp},
\begin{equation}
\label{eq:R}
\psi'(\alpha)^{2}\mathbb{E}_{t}\left[A(k,t)A(l,t)\right]=2e^{2\alpha t}c_{k}(t)c_{l}(t)+e^{\alpha t}\left(4\psi'(\alpha)e^{\alpha t}F(t)c_{k}(t)c_{l}(t)+\frac{R}{\psi'(\alpha)} \right)+\mathcal{O}\left(1\right),
\end{equation}
with
\begin{align*}
R:=&-\psi'(\alpha)\int_{0}^{\infty}2\theta \ \frac{e^{-\theta a}W(a)}{W_{\theta}(a)^{2}}\left(1-\frac{1}{W_{\theta}(a)} \right)^{l-1}\int_{0}^{a}\frac{e^{-\theta s}}{W_{\theta}(s)^{2}}\left(1-\frac{1}{W_{\theta}(a)} \right)^{k-1}dsda\\&-\psi'(\alpha)\int_{0}^{\infty}2\theta \ \frac{e^{-\theta a}W(a)}{W_{\theta}(a)^{2}}\left(1-\frac{1}{W_{\theta}(a)} \right)^{k-1}\int_{0}^{a}\frac{e^{-\theta s}}{W_{\theta}(s)^{2}}\left(1-\frac{1}{W_{\theta}(a)} \right)^{l-1}dsda\\
&+\psi'(\alpha)\int_{0}^{\infty}\theta W(a)^{-1}\left(\mathbb{E}_{t}\left[A(k,t)\mathds{1}_{Z_{0}(a)=l} \right]+\mathbb{E}_{t}\left[A(l,t)\mathds{1}_{Z_{0}(a)=k} \right]\right)da,
\end{align*}
and $F$, $\mu$ are defined in Lemma \ref{lem:WComp}.
Now, using \eqref{eq:survie}, we have
\[
\mathbb{E}_{t}\mathcal{E}^{2}-2=-2\mu\psi'(\alpha)e^{-\alpha t}+o(e^{-\alpha t}),
\]
which leads to
\begin{align*}
\mathbb{E}_{t}&\left[\left(e^{-\alpha t}\psi'(\alpha)A(k,t)-\mathcal{E}c_{k}\right)\left(e^{-\alpha t}\psi'(\alpha)A(l,t)-\mathcal{E}c_{l}\right)\right]\\
=&\mathbb{E}_{t}\left[e^{-2\alpha t}\psi'(\alpha)^{2}A(k,t)A(l,t)\right]-c_{l}\mathbb{E}_{t}\left[e^{-\alpha t}\psi'(\alpha)A(k,t)\mathcal{E}\right]-c_{k}\mathbb{E}_{t}\left[e^{-\alpha t}\psi'(\alpha)A(l,t)\mathcal{E}\right]\\
&+2c_{k}c_{l}-2c_{k}c_{l}\mu\psi'(\alpha)e^{-\alpha t}+o(e^{-\alpha t}),\\
=&2\left(c_{k}(t)-c_{k} \right)\left(c_{l}(t)-c_{l}\right)-4\mu\psi'(\alpha)c_{k}c_{l}e^{-\alpha t}+Re^{-\alpha t}\\&-\left(2c_{k}(t)c_{l}+2c_{l}(t)c_{k}-2c_{k}c_{l}\psi'(\alpha)e^{-\alpha t}\mathbb{E}_{t}N_{t}\mathcal{E} \right)\\&+\psi'(\alpha)c_{l}e^{-\alpha t}\mathbb{E}_{t}\left[\left(c_{k}N_{t}-A(k,t)\right)\mathcal{E}\right]+\psi'(\alpha)c_{k}e^{-\alpha t}\mathbb{E}_{t}\left[\left(c_{l}N_{t}-A(l,t)\right)\mathcal{E}\right]+o(e^{-\alpha t}),\\
\end{align*}
Since, by Lemma \ref{lem: asyComp}
\[
c_{k}(t)=c_{k}+\mathcal{O}(e^{-\theta t})=c_{k}+o(e^{-\alpha t}),
\]
it follows, combining \eqref{eq:AsNtAkate}, \eqref{eq:AsNtE}, and Lemma \ref{lem:akegnt},  that
\begin{align*}
e^{\alpha t}\mathbb{E}_{t}&\left[\left(e^{-\alpha t}\psi'(\alpha)A(k,t)-\mathcal{E}c_{k}\right)\left(e^{-\alpha t}\psi'(\alpha)A(l,t)-\mathcal{E}c_{l}\right)\right]
\\=&
 \psi'(\alpha)\left(c_{k}\gamma_{l}+ c_{l}\gamma_{k}\right)+c_{k}c_{l}\left(2e^{\alpha t}-2\psi'(\alpha)\mathbb{E}_{t}N_{t}\mathcal{E} \right)+R-4\mu\psi'(\alpha)c_{k}c_{l}+o(1)\\
=&
\psi'(\alpha)\left(c_{k}\gamma_{l}+ c_{l}\gamma_{k}\right)+c_{k}c_{l}\left(\frac{1}{\psi'(\alpha)}+3\mu \right)+R-4\mu\psi'(\alpha)c_{k}c_{l}+o(1).\\
\end{align*}
The result follows readily from the fact that $\mathbb{P}\left(N_{t}>0\right)\sim\frac{\alpha}{b}$.

\end{proof}
\begin{lem}[Boundedness of the third moment]
\label{lem:aktBound}
Let $k_{1},k_{2},k_{3}$ three positive integers, then
\[
\mathbb{E}\left[\prod_{i=1}^{3}\left|e^{-\frac{\alpha}{2} t}\left(\psi'(\alpha)A(k_{i},t)-e^{\alpha t}\mathcal{E}c_{k_{i}}\right) \right| \right]=\mathcal{O}\left(1\right).
\]
\end{lem}
\begin{proof}
We have,

\[
\mathbb{E}\left[\left|\prod_{i=1}^{3}\frac{\left(\psi'(\alpha)A(k_{i},t)-e^{\alpha t}\mathcal{E}c_{k_{i}}\right)}{e^{\frac{\alpha}{2} t}} \right| \right]\leq \prod_{i=1}^{3}\left(\mathbb{E}\left[\left|\frac{\left(\psi'(\alpha)A(k_{i},t)-e^{\alpha t}\mathcal{E}c_{k_{i}}\right)}{e^{\frac{\alpha}{2} t}} \right|^{3} \right]\right)^{\frac{1}{3}}.
\]
Hence, we only have to prove the Lemma for $k_{1}=k_{2}=k_{3}=k$.
Hence,
\begin{align*}
\mathbb{E}\left[\left|\frac{\left(\psi'(\alpha)A(k,t)-e^{\alpha t}\mathcal{E}c_{k}\right)}{e^{\frac{\alpha}{2} t}} \right|^{3} \right]&\leq8\mathbb{E}\left[\left|\frac{\psi'(\alpha)A(k,t)-c_{k}N_{t}}{e^{\frac{\alpha}{2} t}} \right|^{3} \right]+ 8c_{k}\mathbb{E}\left[\left|\frac{\psi'(\alpha)N_{t}-N^{\infty}_{t}}{e^{\frac{\alpha}{2} t}} \right|^{3} \right]\\&\qquad\qquad\qquad\qquad\qquad\qquad\qquad\qquad\qquad+8c_{k}\mathbb{E}\left[\left|\frac{N^{\infty}_{t}-e^{\alpha t}\mathcal{E}}{e^{\frac{\alpha}{2} t}} \right|^{3} \right].
\end{align*}
The last two terms have been treated in the proof of \cite[Lemma 6.4]{H1}, and the boundedness of
\[
\mathbb{E}\left[\left|\frac{\psi'(\alpha)A(k,t)-c_{k}N_{t}}{e^{\frac{\alpha}{2} t}} \right|^{3} \right],
\]
follows from the following Lemma \ref{lem: thridBounded} and H\"older's inequality.

\end{proof}


\begin{lem}
\label{lem: thridBounded}For all $k\geq 1$,
\[
\mathbb{E}\left[\left(\frac{A(k,t)-c_{k}N_{t}}{e^{-\frac{\alpha}{2}t}}\right)^{4} \right],
\]
is bounded.
\end{lem}
Due to technicality, the proof of this lemma is postponed to the end in appendix.
\subsubsection{Arbitrary initial distribution case}
The following Lemmas are the counter part of Lemmas 6.5, 6.6, and 6.7 of \cite{H1}. They play the same role in the proof of Theorem \ref{thm:tclAllelic} as in the proof of the central limit theorem given in \cite{H1}.
In the sequel, we denote by $\left(A(k,t,\Xi)\right)_{k\geq 1}$, the frequency spectrum of the splitting tree where the lifetime of the ancestral individual is $\Xi$, in the same manner as for $N_{t}\left(\Xi\right)$ in \cite{H1}.
\begin{equation}
\label{eq:asCV2}
\mathcal{E}_{i}:=\lim\limits_{t\to\infty}\psi'(\alpha)e^{-\alpha t}N^{i}_{t}, \quad a.s,
\end{equation}
and, let $\mathcal{E}\left(\Xi\right)$  be the random variable defined by
\begin{equation}
\label{eq:ancBranlim}
\mathcal{E}\left(\Xi\right):=\int_{[0,\infty]}\mathcal{E}_{(\xi_{u})}e^{-\alpha u} \mathds{1}_{\Xi>u}\ \xi(du).
\end{equation}
\begin{lem}[$L^{2}$ convergence in the general case]
\label{lem:genConvAkt}
Consider the general frequency spectrum  $\left(A(k,t,\Xi)\right)_{k\geq 1}$, then, for all $k$, $\psi'(\alpha)e^{-\alpha t}A(k,t,\Xi)$ converge to $\mathcal{E}\left(\Xi \right)$ (see \ref{eq:ancBranlim}) in $L^{2}$ as $t$ goes to infinity and
\[
\lim\limits_{t\to\infty}e^{-\alpha t}\mathbb{E}\left[\left(\psi'(\alpha)A(k,t,\Xi)-e^{\alpha t}\mathcal{E}(\Xi)c_{k}\right)\left(\psi'(\alpha)A(l,t,\Xi)-e^{\alpha t}\mathcal{E}(\Xi)c_{k}\right)\right]=\frac{\alpha}{b}a_{k,l}\int_{ \mathbb{R}_{+}}e^{-\alpha u}\mathbb{P}\left(\Xi>u \right)bdu,
\]
where the convergence is uniform w.r.t.\ the random variable $\Xi$.
In the case where $\Xi$ is distributed as $O^{(\beta t)}_{2}$, for $0<\beta<\frac{1}{2}$, we get
\[
\lim\limits_{t\to\infty}e^{-\alpha t}\mathbb{E}\left[\left(\psi'(\alpha)A(k,t,O^{(\beta t)}_{2})-e^{\alpha t}\mathcal{E}(O^{(\beta t)}_{2})c_{k}\right)\left(\psi'(\alpha)A(l,t,O^{\beta t}_{2})-e^{\alpha t}\mathcal{E}(O^{(\beta t)}_{2})c_{k}\right)\right]=\psi'(\alpha)a_{k,l}.
\]
\end{lem}
\begin{lem}[First moment]
The first moments are asymptotically bounded, that is, for all $k\geq 1$,
\label{lem:firstMomakate}
\[
\mathbb{E}\left(\psi'(\alpha)A(k,t)(\Xi)-e^{\alpha t}c_{k}\mathcal{E}(\Xi)\right)\leq \mathcal{O}(1),
\]
uniformly with respect to the random variable $\Xi$.
\end{lem}
\begin{lem}[Boundedness in the general case.]
\label{lem:bdgen}
Let $k_{1},k_{2},k_{3}$ three positive integers, then
\[
\mathbb{E}\left[\left|\prod_{i=1}^{3}\frac{\left(\psi'(\alpha)A(k_{i},t)-e^{\alpha t}\mathcal{E}c_{k_{i}}\right)}{e^{\frac{\alpha}{2} t}} \right| \right]=\mathcal{O}\left(1\right),
\]
uniformly with respect to the random variable $\Xi$.
\end{lem}
We do not detail the proofs of these results since they are direct adaptations of the proofs of Lemmas 6.5, 6.6, and 6.7 of \cite{H1}.
\subsection{Proof of the result}
The following result is based on the fact that, in the clonal sub-critical case, the lifetime of a family is expected to be small. It follows that one can expect that all the family of size $k$ live in different subtrees as soon as $t>>u$. This is the point of the following lemma.
\begin{lem}

\label{lem:limitEvent}
Suppose that $\alpha<\theta$.
If we denote by $\Gamma_{u,t}$ the event,
\[
\Gamma_{u,t}=\left\{\text{"there is no family in the population at time t which is older than }u\text{"} \right\},
\]
then, for all $\beta$ in $(0,1-\frac{\alpha}{\theta})$, we have
\[
\lim\limits_{t\to\infty}\mathbb{P}_{\beta t}\left(\Gamma_{\beta t,t} \right)=1.
\]
\end{lem}
\begin{proof}
The proof of this Lemma, as the calculation of the moments of $A(k,t)$ relies on the representation of the genealogy of the living population at time $t$ as a coalescent point process \cite{CH}. Moreover, we denote by $\widetilde{N}^{(t)}_{u}$ the number of living individuals at time $u$ who have alive descent at time $t$. In \cite{CH}, we showed that, under $\mathbb{P}_{t}$, $\widetilde{N}^{(t)}_{u}$ is geometrically distributed with parameter $\frac{W(t-u)}{W(t)}$.

Now, $\mathds{1}_{\Gamma_{u,t}}$ can be rewritten as
\[
\mathds{1}_{\Gamma_{u,t}}=\prod_{i=1}^{\widetilde{N}^{(t)}_{u}}\mathds{1}_{\left\{Z^{i}_{0}(t-u)=0 \right\}},
\]
where $Z_{0}^{i}(t-u)$ denotes the number of individuals alive at time $t$ descending from the $i$th individual alive at time $u$ and carrying its type (the clonal type of the sub-CPP). Moreover, from Proposition 4.3 of \cite{CH}, we know that that under $\mathbb{P}_{t}$, the family $Z^{(i)}_{0}(t-u)$ is an i.i.d.\ family of random variables distributed as $Z_{0}(t-u)$ under $\mathbb{P}_{t-u}$, and $\widetilde{N}^{(t)}_{u}$ is independent of $Z^{(i)}_{0}(t-u)$ (still under $\mathbb{P}_{t}$).

Then,
\[
\mathbb{P}_{t}\left(\Gamma_{t,u} \right)=\mathbb{E}_{t}\left[\mathbb{P}_{t-u}\left(Z_{0}(t-u)=0\right)^{\widetilde{N}^{(t)}_{u}} \right]
=\frac{\mathbb{P}_{t-u}\left(Z_{0}(t-u)=0\right)\frac{W(t-u)}{W(t)}}{1-\mathbb{P}_{t-u}\left(Z_{0}(t-u)=0\right)\left(1-\frac{W(t-u)}{W(t)}\right)}.
\]
Using \eqref{eq:loizzero}, some calculus leads to, 
\[
\mathbb{P}_{t}\left(\Gamma_{t,u} \right)=1-\frac{1}{1+\frac{W_{\theta}(t-u)}{e^{-\theta (t-u)}W(t)}\left(1-\frac{e^{-\theta (t-u)}W(t-u)}{W_{\theta}(t-u)} \right)}.
\]
Now, since,
\[
\mathbb{P}_{t}\left(\Gamma_{t,u} \right)=\mathbb{P}_{u}\left(\Gamma_{t,u} \right)\frac{\mathbb{P}\left(N_{u}>0 \right)}{\mathbb{P}\left(N_{t}>0 \right)}+\frac{\mathbb{P}\left(\Gamma_{t,u},N_{t}=0,N_{u}>0 \right)}{\mathbb{P}\left(N_{t}>0 \right)},
\]
taking $u=\beta t$, we obtain, using Lemma \ref{lem: asyComp} and
\[
\mathbb{P}\left(N_{t}=0,N_{\beta t}>0 \right)=\mathbb{P}\left(N_{\beta t}>0\right)-\mathbb{P}\left(N_{t}>0\right)\underset{t\to\infty}{\to}0,
\]
 the desired result.
\end{proof}
\begin{proof}[Proof of Theorem \ref{thm:tclAllelic}]
Fix $0<u<t$. Note that the event  $\Gamma_{u,t}$ of Lemma \ref{lem:limitEvent} can be rewritten as
\begin{equation}
\label{eq:decEvent}
\mathds{1}_{\Gamma_{u,t}}=\prod_{i=1}^{N_{u}}\mathds{1}_{\left\{Z^{i}_{0}(t-u,O_{i})=0 \right\}},
\end{equation}
where $Z_{0}^{i}(t-u,O_{i})$ denote the number of individuals alive at time $t$ carrying the same type as the $i$th alive individual at time $u$, that is the ancestral family of the splitting constructed from the residual lifetime of the $i$th individual (see Section 4 in \cite{H1}).

Let $K$ be a multi-integer, we denote by $\mathcal{L}^{(K)}$ (resp. $A(K,t)$) the random vector $\left(\mathcal{L}^{k_{1}},\dots,\mathcal{L}^{k_{N}}\right)$ (resp. $\left(A(k_{1},t),\dots,A(k_{N},t)\right)$) with
\[
\mathcal{L}_{t}^{k_{i}}=\frac{\psi'(\alpha)A(k,t)-c_{k}e^{\alpha t}\mathcal{E}}{e^{\frac{\alpha}{2}t}}.
\]
On the event $\Gamma_{u,t}$, we have a.s.,
\[
A(k_{l},t)=\sum_{i=1}^{N_{u}}A^{(i)}(k_{l},t-u,O_{i}), \quad \forall l=1,\dots,N,
\]
where the family $\left(A^{(i)}\left(k_{l},t-u,O_{i}\right)\right)_{i\geq 1}$  stand for the frequency spectrum for each subtree, which are independent from Lemma \ref{lem:residual}.
Hence, using Lemma \ref{lem:dec},
\[
\mathcal{L}_{t}^{k_{l}}=\sum_{i=1}^{N_{u}}\frac{\psi'(\alpha)A^{(i)}(k_{l},t-u,O_{i})-e^{\alpha(t-u)}\mathcal{E}_{i}(O_{i})c_{k_{l}}}{e^{\frac{\alpha}{2}u}e^{\frac{\alpha}{2}(t-u)}}.
\]
By Lemma \ref{lem:residual}, that the family $\left(A^{i}(k_{l},t-u,O_{i}) \right)_{2\leq i\leq N_{u}}$ is i.i.d. under $\mathbb{P}_{u}$.

In the sequel, we denote, for all $l$ and $i\geq 1$,
\[
\tilde{A}^{(i)}\left(k_{l},t-u,O_{i}\right)=\frac{\psi'(\alpha)A^{(i)}\left(k_{l},t-u,O_{i}\right)-e^{\alpha(t-u)}\mathcal{E}_{i}(O_{i})c_{k_{l}}}{e^{\frac{\alpha}{2}(t-u)}}.
\]
Now, let 
\begin{align*}
\varphi_{K}\left(\xi\right)&:=\mathbb{E}\left[\exp\left(i<\tilde{A}\left(K,t-u,O_{2}\right),\xi>\right)\mathds{1}_{Z^{2}_{0}(t-u,O_{2})=0}\right],\\
\tilde{\varphi}_{K}\left(\xi\right)&:=\mathbb{E}\left[\exp\left(i<\tilde{A}\left(K,t-u,O_{1}\right),\xi>\right)\mathds{1}_{Z^{1}_{0}(t-u,O_{1})=0}\right].
\end{align*}

From this point, following closely the proof of Theorem 3.2 of \cite{H1}. Taking $u=\beta$ in $\left(0, \frac{1}{2}\wedge (1-\frac{\alpha}{\theta}) \right)$, the only difficulty is to handle the indicator function $\mathds{1}_{Z_{0}(t-u,O_{i})>0}$ in the Taylor development of $\varphi_{K}$. We show how it can be done for one of the second order terms, and leave the rest of the details to the reader.

It follows from Hölder's inequality that
\begin{multline}
\label{eq:note}
\mathbb{E}\left[\left(
\frac{\psi'(\alpha)A^{(i)}\left(k_{l},(1-\beta)t,O_{i}\right)-e^{\alpha((1-\beta)t)}\mathcal{E}_{i}(O_{i})c_{k_{l}}}{e^{\frac{\alpha}{2}((1-\beta)t)}}
 \right)^{2}\mathds{1}_{Z^{2}_{0}((1-\beta)t,O_{2})>0} \right]\\\leq\mathbb{E}\left[\left(
 \frac{\psi'(\alpha)A^{(i)}\left(k_{l},(1-\beta)t,O_{i}\right)-e^{\alpha(1-\beta)t}\mathcal{E}_{i}(O_{i})c_{k_{l}}}{e^{\frac{\alpha}{2}(1-\beta)t}}
  \right)^{3} \right]^{\frac{2}{3}}\mathbb{P}\left(Z^{2}_{0}((1-\beta)t,O_{2})>0\right)^{\frac{1}{3}},
\end{multline}
from which it follows, using Lemma \ref{lem:bdgen},  that the r.h.s.\ of this last inequality is $\mathcal{O}\left(\mathbb{P}\left(Z^{2}_{0}(t-u,O_{2})>0\right)^{\frac{1}{3}} \right)$.
Now, using \eqref{eq:decEvent} and Lemma \ref{lem:limitEvent}, it is easily seen that
\[
\lim\limits_{t\to\infty}\mathbb{P}\left(Z^{2}_{0}((1-\beta)t,O_{2})>0\right)=0.
\]
Finally, using Lemma \ref{lem:quadConv}, we get
\[
\lim\limits_{t\to\infty}\mathbb{E}\left[\left(
\frac{\psi'(\alpha)A^{(i)}\left(k_{l},t-u,O_{i}\right)-e^{\alpha(t-u)}\mathcal{E}_{i}(O_{i})c_{k_{l}}}{e^{\frac{\alpha}{2}(t-u)}}
 \right)^{2}\mathds{1}_{Z^{2}_{0}(t-u,O_{2})=0} \right]=\psi'(\alpha)a_{k,k}.
\]
These allow us to conclude that
\[
\lim\limits_{t\to\infty} \mathbb{E}_{\beta t}\left[e^{i<\mathcal{L}^{(K)}_{t},\xi>}\mathds{1}_{\Gamma_{t}} \right]=\frac{1}{1+\sum_{i,j=1}^{N}\mathcal{M}_{i,j}\ \xi_{i}\xi_{j}},
\]
where $K_{i,j}$ is given by
\[
\mathcal{M}_{i,j}:=\psi'(\alpha)a_{K_{i},K_{j}},
\]
with $K$ is the multi-integer $(k_{1},\dots,k_{N})$, and the $a_{l,k}$s are defined
 in Lemma \ref{lem:quadConv}.

To end the proof, note that, \[
\left|\mathbb{E}_{\infty}\left[e^{i<\mathcal{L}^{(K)}_{t},\xi>} \right]-\mathbb{E}_{\beta t}\left[e^{i<\mathcal{L}^{(K)}_{t},\xi>}\mathds{1}_{\Gamma_{\beta t,t}} \right]\right|\leq \mathbb{E}\left[\left|\frac{\mathds{1}_{\text{NonEx}}}{\mathbb{P}\left(\text{NonEx} \right)} -\frac{\mathds{1}_{N_{\beta t}>0}\mathds{1}_{\Gamma_{\beta t,t}}}{\mathbb{P}\left(N_{\beta t}>0 \right)}\right| \right]\underset{t\to \infty}{\to}0,
\]
thanks to Lemma \ref{lem:limitEvent}.
\end{proof}
\section{Proof of Theorem \ref{thm:cltFin}}
\label{proof:th3}
Since all the ideas of the proof of this theorem have been developed the preceding sections, we do not detail all the proof. The only step which needs clarification is the computation of the covariance matrix of the Laplace limit law $\mathcal{M}$. According to the proof of Theorem \ref{thm:tclAllelic}, it is given by
\begin{align*}
\mathcal{M}_{i,j}:=&\lim\limits_{t\to\infty}\frac{W(\beta t)}{e^{\alpha\beta t}}\mathbb{E}\Bigg[\left(
\frac{\psi'(\alpha)A^{(i)}\left(k_{i},(1-\beta)t,O_{i}\right)-\psi'(\alpha)c_{k_{i}}N_{(1-\beta)t}}{e^{\frac{\alpha}{2}((1-\beta)t)}}
 \right)\\&\times\left(
 \frac{\psi'(\alpha)A^{(i)}\left(k_{j},(1-\beta)t,O_{i}\right)-c_{k_{j}}N_{(1-\beta)t}}{e^{\frac{\alpha}{2}((1-\beta)t)}}
  \right)\mathds{1}_{Z^{2}_{0}((1-\beta)t,O_{2})>0} \Bigg],
\end{align*}
which is equal, thanks to \eqref{eq:note} and an easy adaptation of Lemma 6.6 in \cite{H1}, to
\begin{align*}
\mathcal{M}_{i,j}=&\lim\limits_{t\to\infty}\frac{b\psi'(\alpha)}{\alpha}\frac{W(\beta t)}{e^{\alpha\beta t}}e^{\alpha t}\mathbb{E}\left[\left(e^{-\alpha t}A(k_{i},t)-c_{k_{i}}e^{-\alpha t}N_{t} \right)\left(e^{-\alpha t}A(k_{j},t)-c_{k_{j}}e^{-\alpha t}N_{t} \right) \right].
\end{align*}
So it remains to get the limit of
\[
e^{\alpha t}\mathbb{E}\left[\left(e^{-\alpha t}\psi'(\alpha)A(k,t)-\psi'(\alpha)c_{k}e^{-\alpha t}N_{t} \right)\left(e^{-\alpha t}\psi'(\alpha)A(l,t)-c_{l}e^{-\alpha t}\psi'(\alpha)N_{t} \right) \right],
\]
as $t$ goes to infinity.
We recall that using the calculus made in the proof of Theorem 6.3 of \cite{CH}, we have
\begin{equation}
\label{eq:DerEq}
\mathbb{E}_{t}A(k,t)N_{t}=2W(t)^{2}c_{k}(t)-2W(t)\int_{[0,t]}\theta\mathbb{P}_{a}\left(Z_{0}(a)=k \right)da
+W(t)\int_{[0,t]}\theta W(a)^{-1}\mathbb{E}_{a}\left[N_{a}\mathds{1}_{Z_{0}(a)=k}\right]da.
\end{equation}
Moreover, 
\eqref{eq:R} entails
\[
\psi'(\alpha)^{2}\mathbb{E}_{t}A(k,t)A(l,t)=2W(t)^{2}c_{k}(t)c_{l}(t)+RW(t)+o(e^{-\alpha t}),
\]
with
\begin{align*}
R:=&-\psi'(\alpha)\int_{0}^{\infty}2\theta W(a)^{-1}\mathbb{P}_{a}\left(Z_{0}(a)=k\right)\mathbb{E}_{a}\left[A(l,a)\right]da\\&\quad\quad\quad+\psi'(\alpha)\int_{0}^{\infty}2\theta W(a)^{-1}\mathbb{P}_{a}\left(Z_{0}(a)=l\right)\mathbb{E}_{a}\left[A(k,a)\right]da\\
&\quad\quad\quad+\psi'(\alpha)\int_{0}^{\infty}\theta W(a)^{-1}\left(\mathbb{E}_{t}\left[A(k,t)\mathds{1}_{Z_{0}(a)=l} \right]+\mathbb{E}_{t}\left[A(l,t)\mathds{1}_{Z_{0}(a)=k} \right]\right)da.
\end{align*}
These identities allow us to obtain
\begin{align*}
&\mathbb{E}_{t}\left[\left(A(k,t)-c_{k}N_{t} \right)\left(A(l,t)-c_{l}N_{t}\right) \right]=2W(t)^{2}c_{k}(t)c_{l}(t)+e^{-\alpha t}R+o(e^{-\alpha t}),\\
&-2c_{l}c_{k}(t)W(t)^{2}+2c_{l}W(t)\int_{[0,t]}\theta\mathbb{P}_{a}\left(Z_{0}(a)=k \right)da
-c_{l}W(t)\int_{[0,t]}\theta W(a)^{-1}\mathbb{E}_{a}\left[N_{a}\mathds{1}_{Z_{0}(a)=k}\right]da\\
&-2c_{k}c_{l}(t)W(t)^{2}+2c_{l}W(t)\int_{[0,t]}\theta\mathbb{P}_{a}\left(Z_{0}(a)=l \right)da
-c_{k}W(t)\int_{[0,t]}\theta W(a)^{-1}\mathbb{E}_{a}\left[N_{a}\mathds{1}_{Z_{0}(a)=l}\right]da\\
&+c_{k}c_{l}W(t)^{2}\left(2-\frac{1}{W(t)} \right)\\
&=2W(t)^{2}\left(c_{k}(t)-c_{l}\right)\left(c_{l}(t)-c_{k}\right)+e^{-\alpha t}\frac{R}{\psi'(\alpha)}+o(e^{-\alpha t}),\\
&+2c_{l}W(t)\int_{[0,t]}\theta\mathbb{P}_{a}\left(Z_{0}(a)=k \right)da
-c_{l}W(t)\int_{[0,t]}\theta W(a)^{-1}\mathbb{E}_{a}\left[N_{a}\mathds{1}_{Z_{0}(a)=k}\right]da\\
&+2c_{l}W(t)\int_{[0,t]}\theta\mathbb{P}_{a}\left(Z_{0}(a)=l \right)da
-c_{k}W(t)\int_{[0,t]}\theta W(a)^{-1}\mathbb{E}_{a}\left[N_{a}\mathds{1}_{Z_{0}(a)=l}\right]da\\
&-c_{k}c_{l}W(t).
\end{align*}
Taking the limit as $t$ goes to infinity leads to
\begin{align}
M_{k,l}:=\lim\limits_{t\to\infty}&\psi'(\alpha)^{2}e^{-\alpha t}\mathbb{E}_{t}\left[\left(A(k,t)-c_{k}N_{t} \right)\left(A(l,t)-c_{l}N_{t}\right) \right]=R\nonumber\\
&+2\psi'(\alpha)c_{l}\int_{[0,\infty]}\theta\mathbb{P}_{a}\left(Z_{0}(a)=k \right)da
-\psi'(\alpha)c_{l}\int_{[0,\infty]}\theta W(a)^{-1}\mathbb{E}_{a}\left[N_{a}\mathds{1}_{Z_{0}(a)=k}\right]da\nonumber\\
&+2\psi'(\alpha)c_{l}\int_{[0,\infty]}\theta\mathbb{P}_{a}\left(Z_{0}(a)=l \right)da
-\psi'(\alpha)c_{k}\int_{[0,\infty]}\theta W(a)^{-1}\mathbb{E}_{a}\left[N_{a}\mathds{1}_{Z_{0}(a)=l}\right]da\nonumber\\
&-\psi'(\alpha)c_{k}c_{l}\label{eq:mlk}.
\end{align}
Finally, since
$\mathbb{P}\left(N_{t}>0\right)\sim\frac{\alpha}{b}$,
\[
\mathcal{M}_{i,j}=M_{k_{i},k_{j}}.
\]
\section{Markovian cases}
\label{proof:th4}
 We can get more information on the unknown covariance matrix $K$ in the case where the life duration distribution is exponential. Our study also cover the case $\mathbb{P}_{V}=\delta_{\infty}$ (Yule case), although it does not fit the conditions required by the Theorem \ref{thm:tclAllelic}. The reason comes from our method of calculation for $\mathbb{E}\left[A(k,t)\mathcal{E}\right]$. 
Let us consider the filtration $\left(\mathcal{F}_{t}\right)_{t\in\mathbb{R}_{+}}$, where $\mathcal{F}_{t}$ is the $\sigma$-field generated by the tree truncated above $t$ and the restriction of the mutation measure on $[0,t)$.

Then $N_{t}$ is Markovian with respect to $\mathcal{F}_{t}$ and for all positive real numbers $t\leq s$,
\begin{align*}
\mathbb{E}\left[A(k,t)N_{s}\mid\mathcal{F}_{t} \right]=A(k,t)N_{t}\mathbb{E}\left[N_{s-t}\right].
\end{align*}
So that,
\begin{align*}
\mathbb{E}\left[A(k,t)N_{s} \right]=\mathbb{E}\left[A(k,t)N_{t}\right]\left(W(s-t)-\mathbb{P}_{V}\star W(s-t) \right).
\end{align*}
By making a renormalization by $e^{-\alpha s}$ and taking the limit as $s$ goes to infinity, we get,
\[
\mathbb{E}\left[A(k,t)\mathcal{E} \right]=\psi'(\alpha)e^{-\alpha t}\mathbb{E}\left[A(k,t)N_{t}\right],
\]
since, in the Markovian case, it is known from \cite{CLR} that
\[
\frac{\alpha}{b}=\psi'(\alpha).
\]
Suppose first that $d>0$.
It follows
 that,
\begin{align*}
\mathbb{E}&\left[\left(\psi'(\alpha)A(k,t)-e^{\alpha t}c_{k}\mathcal{E}\right)\left(\psi'(\alpha)A(l,t)-e^{\alpha t}c_{l}\mathcal{E}\right) \right]
=\psi'(\alpha)^{2}\mathbb{E}_{t}\left[A(k,t)A(l,t)\right]\mathbb{P}\left(N_{t}>0\right)\\
&-c_{k}\psi'(\alpha)^{2}\mathbb{E}_{t}\left[A(l,t)N_{t}\right]\mathbb{P}\left(N_{t}>0\right)-c_{l}\psi'(\alpha)^{2}\mathbb{E}_{t}\left[A(k,t)N_{t}\right]\mathbb{P}\left(N_{t}>0\right)\\
&+2\psi'(\alpha)e^{2\alpha t}c_{k}c_{l}
\end{align*}
By \eqref{eq:survie},
\[
\mathbb{P}\left(N_{t}>0\right)=\psi'(\alpha)+\psi'(\alpha)^{2}\mu e^{-\alpha t}+o(e^{-\alpha t}),
\]
so
\begin{multline*}
\mathbb{E}\left[\left(\psi'(\alpha)A(k,t)-e^{\alpha t}c_{k}\mathcal{E}\right)\left(\psi'(\alpha)A(l,t)-e^{\alpha t}c_{l}\mathcal{E}\right) \right]\\=\mathbb{P}\left(N_{t}>0\right)\psi'(\alpha)^{2}\mathbb{E}_{t}\left[\left(A(k,t)-c_{k}N_{t} \right)\left(A(l,t)-c_{l}N_{t} \right)\right]+c_{k}c_{l}\psi'(\alpha)\left(2e^{2\alpha t}-\psi'(\alpha)\mathbb{E}_{t}\left[N_{t}^{2}\right]\mathbb{P}\left(N_{t}>0\right)\right).
\end{multline*}
Finally, since, using Proposition \ref{lem:WComp},
\[
\lim\limits_{t\to\infty}e^{-\alpha t}\left(2e^{2\alpha t}-\psi'(\alpha)\mathbb{E}_{t}\left[N_{t}^{2}\right]\mathbb{P}\left(N_{t}>0\right)\right)=\psi'(\alpha)\left(1-6\mu \right),\]
it follows from \eqref{eq:mlk},
\begin{align*}
\lim\limits_{t\to\infty}\mathbb{E}&\left[\left(\psi'(\alpha)A(k,t)-e^{\alpha t}c_{k}\mathcal{E}\right)\left(\psi'(\alpha)A(l,t)-e^{\alpha t}c_{l}\mathcal{E}\right) \right]
\\&=\psi'(\alpha)M_{k,l}+c_{k}c_{l}\psi'(\alpha)^{2}\left(1-6\mu \right)=\psi'(\alpha)M_{k,l}+c_{k}c_{l}\psi'(\alpha)^{2}\left(1-6\frac{d}{\alpha}\right),
\end{align*}
using that $\mu=\frac{1}{b\mathbb{E}V-1}$.
In the Yule case, an easy adaptation of the preceding proof leads to
\begin{align*}
\lim\limits_{t\to\infty}\mathbb{E}\left[\left(\psi'(\alpha)A(k,t)-e^{\alpha t}c_{k}\mathcal{E}\right)\left(\psi'(\alpha)A(l,t)-e^{\alpha t}c_{l}\mathcal{E}\right) \right]=M_{k,l}+c_{k}c_{l}.
\end{align*}
\section{Numerical studies}
\label{sec:num}
The purpose of this section is to analyze our approximation method and the estimation of the error by virtue of numerical experiments. There are several practical difficulties appearing when one tries to perform such study. 

The first problem, which involves no conceptual difficulties, lies only on the implementation of the formulas appearing in Theorems \ref{thm:tclAllelic}, \ref{thm:tclexp} and \ref{thm:cltFin}. In particular, the computation of the moments of type $\mathbb{E}[A(k,t)\mathds{1}_{Z_{0}(t)=l}]$ are particularly complicated (see Proposition 5.4 in \cite{CH}).
\medskip

Another difficulty is to obtain numerical approximations of the scale functions $W$ and $W_{\theta}$.  For instance, these functions appear in the computation of the covariance matrix of Theorems \ref{thm:tclexp} and \ref{thm:cltFin} or when one wants to simulate the \emph{coalescent point process}. To obtain such approximations, we need to apply numerically the Laplace inversion operator to the functions $\frac{1}{\psi}$ and $\frac{1}{\psi_{\theta}}$.

Unfortunately, the Laplace numerical inversion is a rather difficult problem (see for instance \cite{laplace1} or \cite{laplace2}) which is often computationally expensive. As a consequence, the computational cost of performing multiple numerical integration involving $W$ or $W_{\theta}$ can be  important when done with a crude method. Moreover, these methods presents rough numerical instabilities when the original function is exponentially increasing (inverting $\lambda\to\frac{1}{1-\lambda}$, whose inverse is $x\to e^{x}$, is already a tough numerical problem). 

\medskip

For all these reasons, we provide with this work a Matlab toolbox which handle all these difficulties and allows users to perform numerical experiments without having to take care of these issues.

\medskip

In this whole section, we are interested in the approximation of the frequency spectrum at a fixed time $t$ by the sequence $N_{t}(c_{k})_{k\geq1}$ (we recall that $c_{k}$ was defined in equation \eqref{eq:ck}). As a consequence, the error in this approximation are computed thanks to Theorem \ref{thm:cltFin}. The parameters of the model are set as follows:
\begin{itemize}
	\item $\mathbb{P}_{V}$ is a Rice distribution with shape parameter 1 and scale parameter 1.
	\item $b=1$.
	\item $\theta=1$.
\end{itemize}
For such parameters $\alpha$ approximately equals to $0.5$. Figure \ref{fig:evol} shows the evolution of the frequency spectrum (for $k$ between $1$ and $10$) through time. The different quantities seem to growth exponentially with rate $\alpha$ with a time-shift which depend on $k$. An interesting open question would be to understand the behavior of these shifts.
\begin{figure}
	\includegraphics[scale=0.4,trim= 100 0 20 0mm]{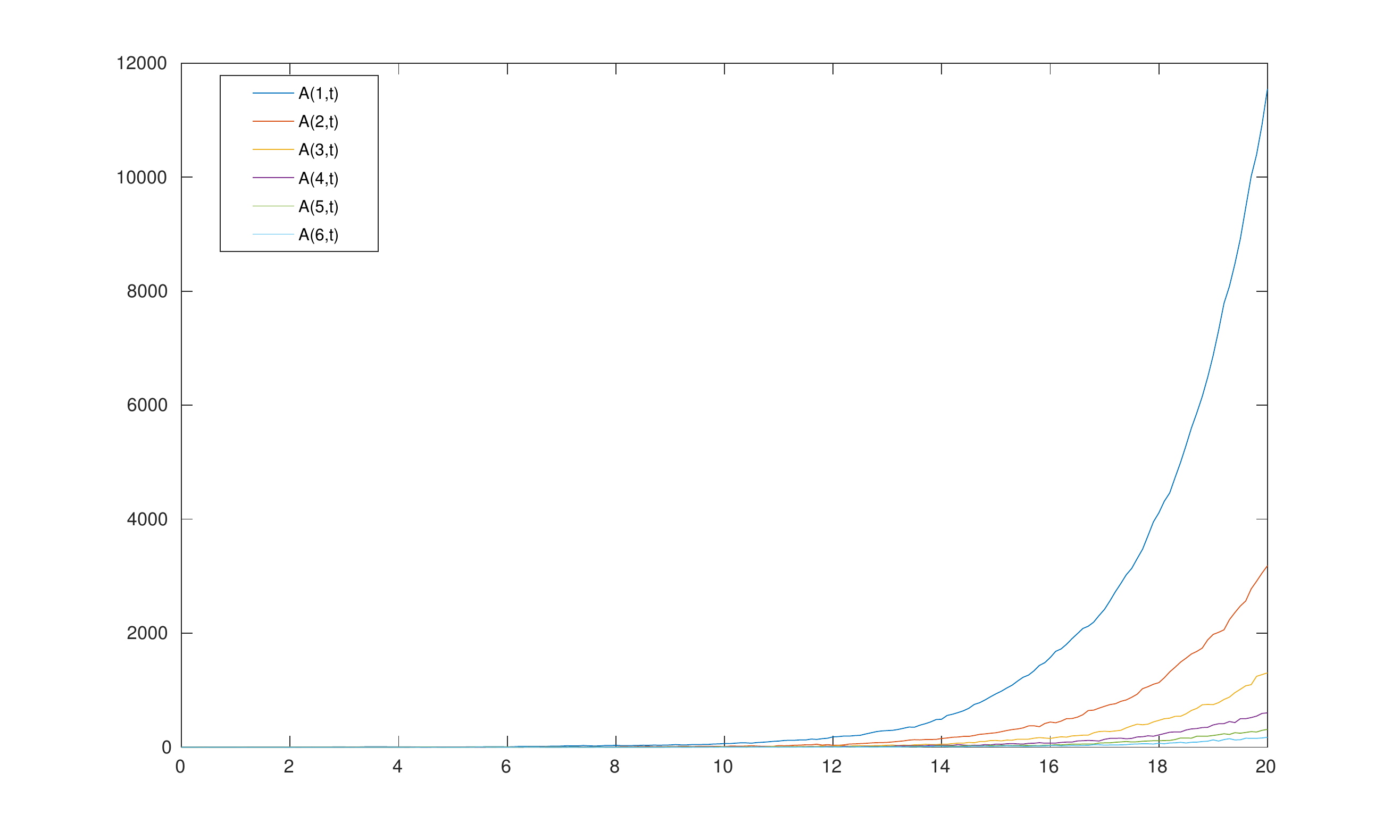}\includegraphics[scale=0.4, trim= 50 0 0 0mm]{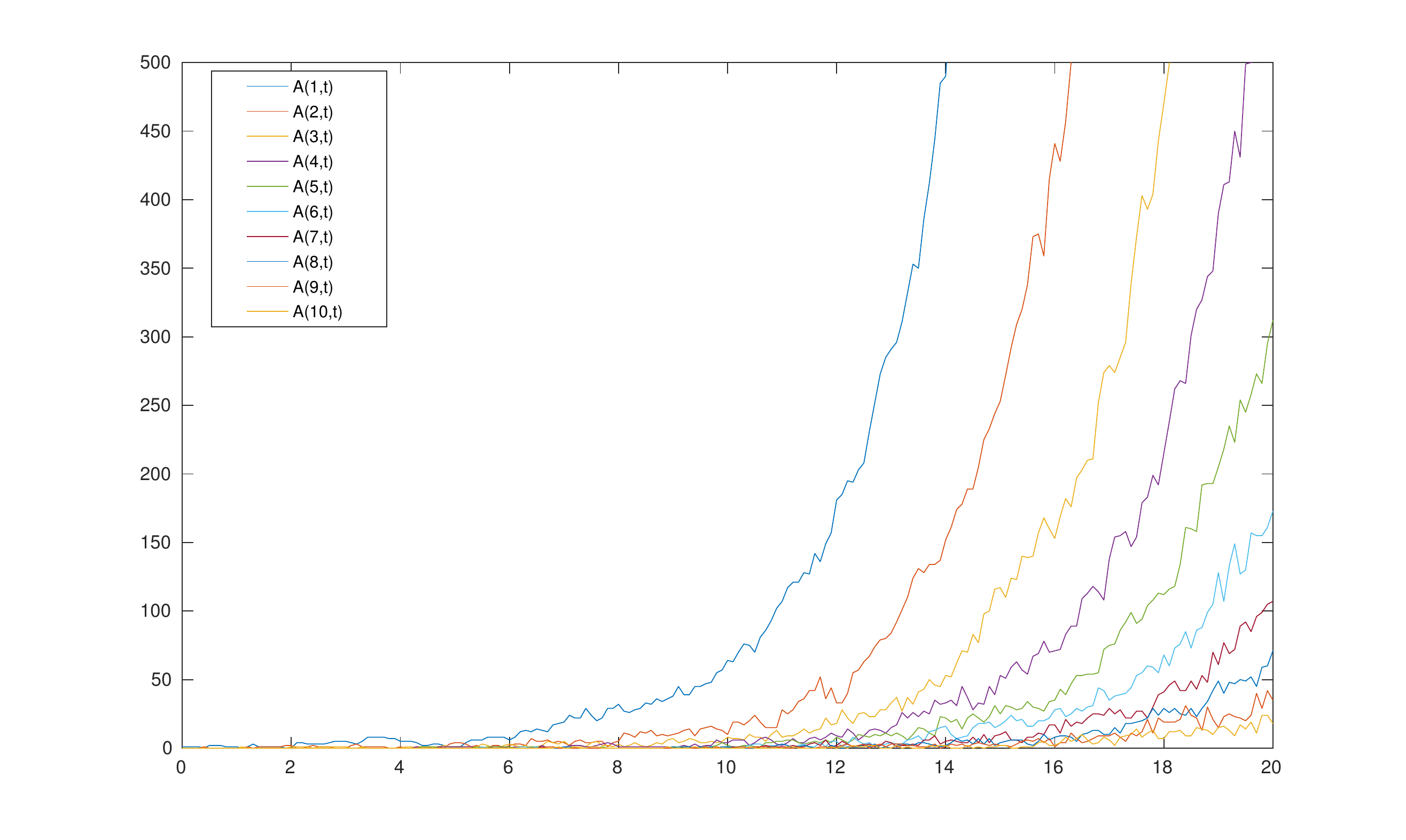}
	\caption{A simulation of the evolution of the frequency spectrum under the given model.}
	\label{fig:evol}
\end{figure}
In order to stress our methods of approximation, the first idea is to look to the renormalized frequency spectrum $\left(\frac{A(k,t)}{c_k}\right)_{k\geq1}$ which is expected to look like $(N_t,\ t \in\mathbb{R}_{+})$. As showed in Figure \ref{fig:evolDynam}, the approximation seems to be quite accurate for $k=1,2$. However, a more quantitative analysis is required. 
Figure \ref{fig:error} shows the absolute error in the approximation of $A(1,t)$ by $c_1 N_t$. This error is a little disappointing since it since to diverge when $t$ goes to infinity. However, even if, according to Figure \ref{fig:error}, the absolute error at time $20$ if of order $10^3$, the relative error shows that this error is quite small with respect to the value of $A(1,20)$.
\begin{figure}
	\includegraphics[scale=0.5, trim = -70 50 50 20 mm]{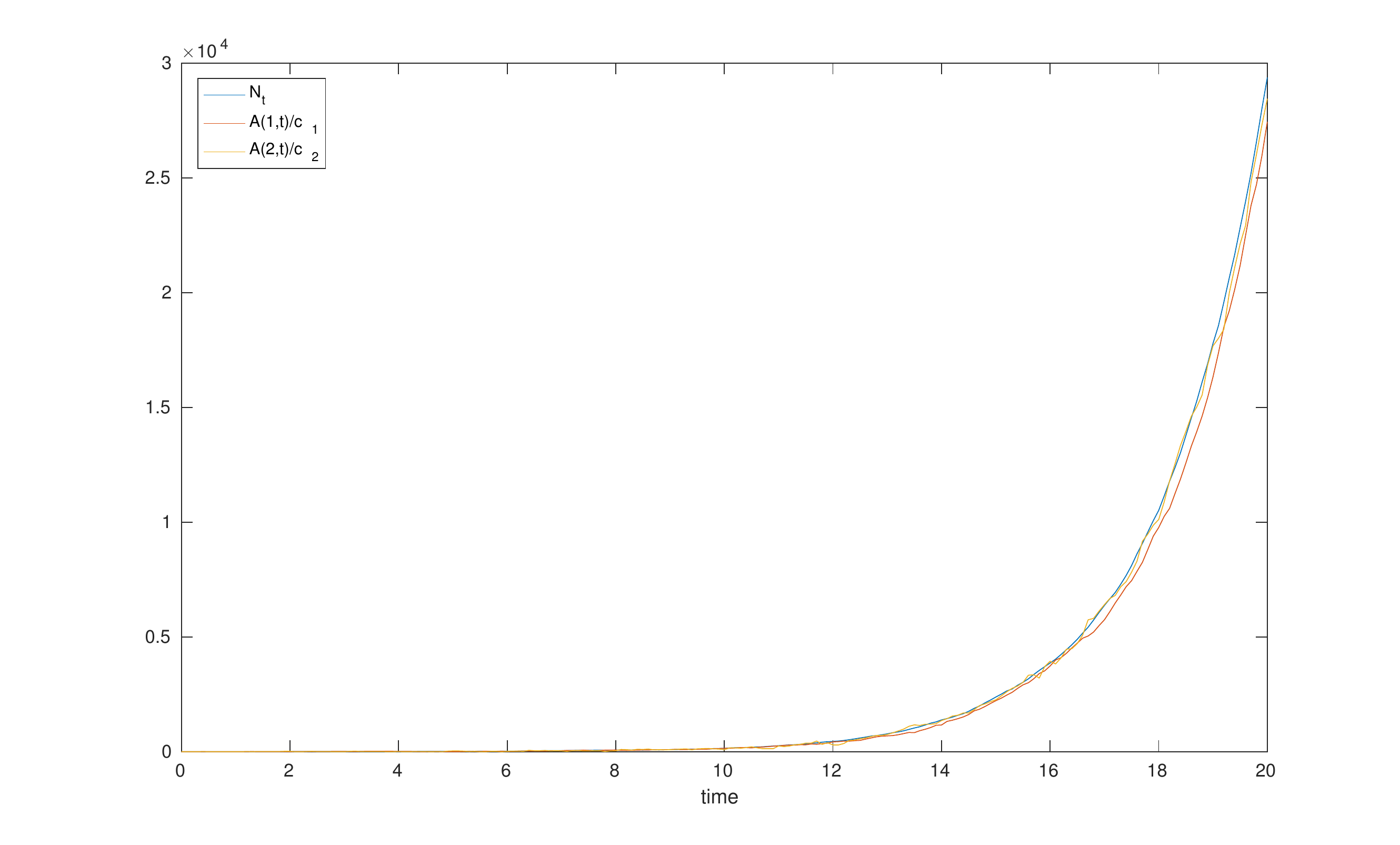}
	\caption{Evolution of the renormalized frequency spectrum $(A(k,t)/c_k)_{k\geq 1}$ under the given model.}
	\label{fig:evolDynam}
\end{figure}
\begin{figure}
	\includegraphics[scale=0.4,trim= 100 0 0 0mm]{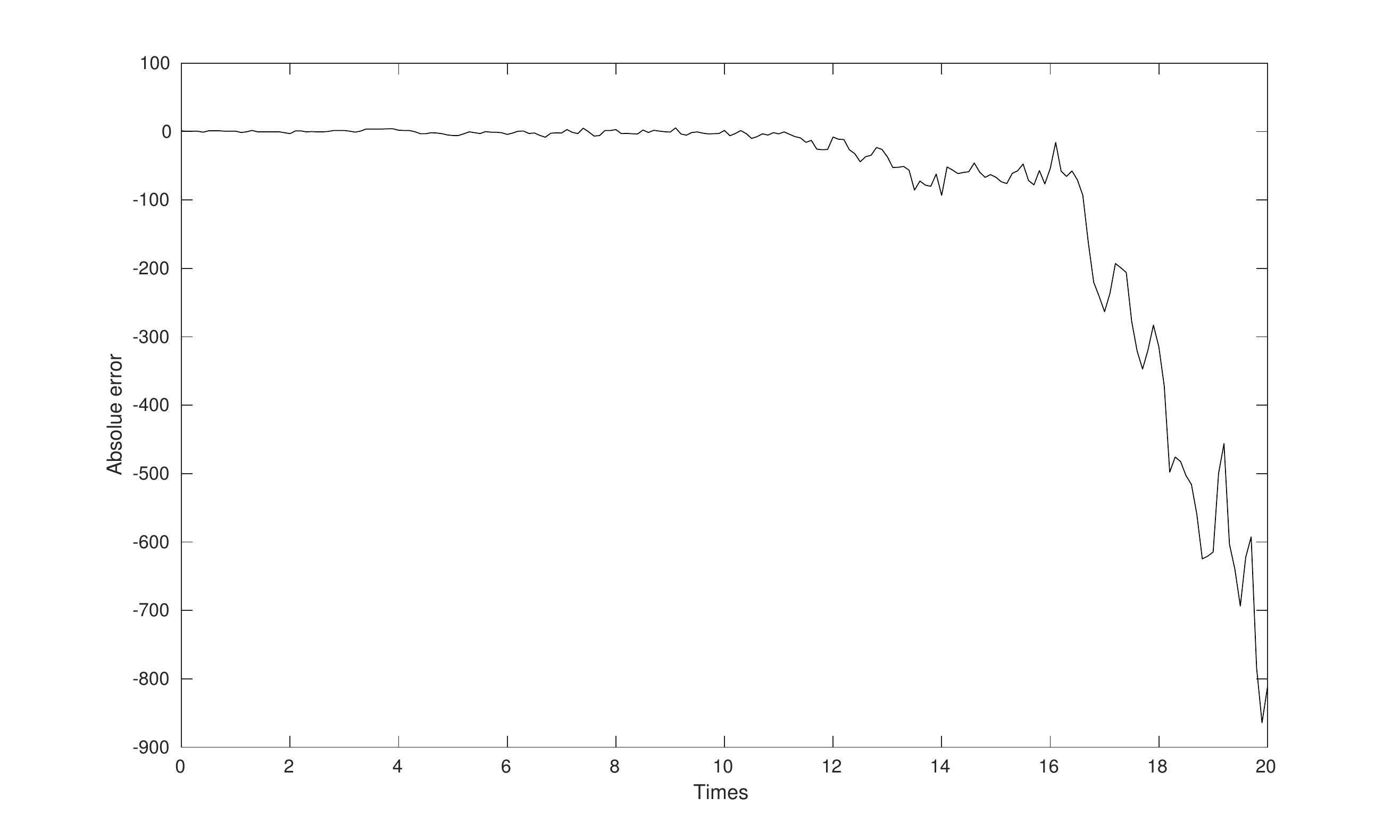}\includegraphics[scale=0.4,trim= 60 0 20 0mm]{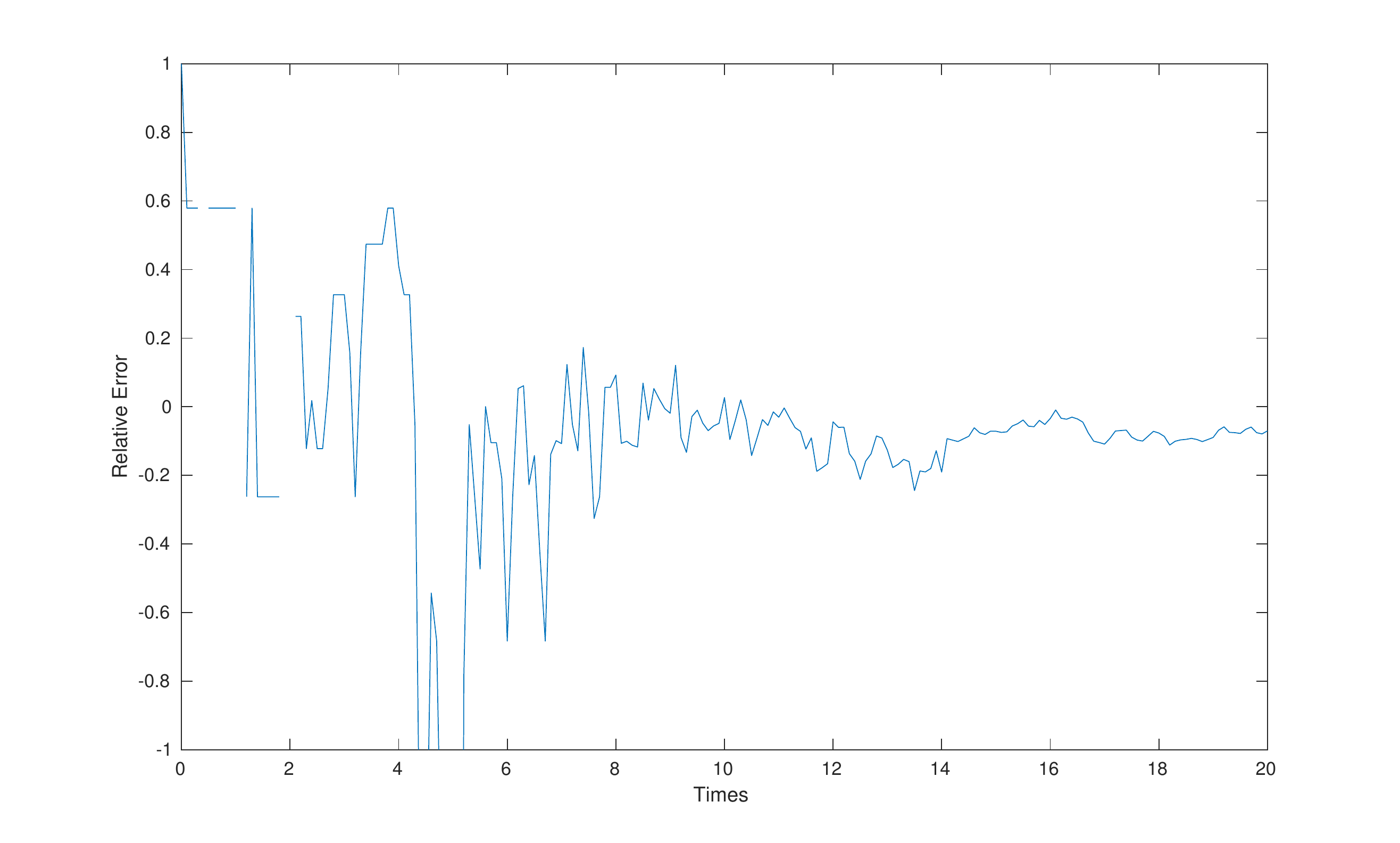}
	\caption{Absolute error (left picture) and relative (right picture) in the approximation of $A(1,t)$ by $c_1 N_t$.}
	\label{fig:error}
\end{figure}
Another question is about the speed of convergence in the central limit theorem stated in Theorem \ref{thm:cltFin}. The red curve of Figure \ref{fig:tcl1} shows the density of the Laplace distribution given in Theorem \ref{thm:cltFin} in the case of $A(1,t)$ whereas the blue histogram shows the distribution of $\psi'(\alpha)(e^{\alpha\frac{t}{2}}(A(1,t)-c_{k}N_{t}))$ for $t=10$ ($\alpha t\sim 5$ and $\mathbb{E}_t [N_t]\sim 300$) from 10000 simulations. This Figure highlights the fact that even if the taken time $t$ is quite small the distribution $\psi'(\alpha)(e^{\alpha\frac{t}{2}}(A(1,t)-c_{k}N_{t}))$ seems already close to the limiting distribution. Figure \ref{fig:tcl2} shows the same kind of behavior in the multidimensional case. To be more quantitative, Figure \ref{fig:L2err} shows the evolution in time of the distance between the density of limit distribution given in Theorem \ref{thm:cltFin} and a kernel estimation of the distribution of $\psi'(\alpha)(e^{\alpha\frac{t}{2}}(A(1,t)-c_{k}N_{t}))$ (the estimation is made from $10000$ simulations at each time). This suggest an exponential rate of convergence in Theorem \ref{thm:cltFin}. In the view of Figure \ref{fig:L2err}, one may think that Berry-Essen type results for Theorem \ref{thm:cltFin} would be quite interesting, in particular to understand how the speed of convergence is related to choice of the parameters.
\begin{figure}
	\includegraphics[scale=0.5, trim= -70 50 50 0 mm]{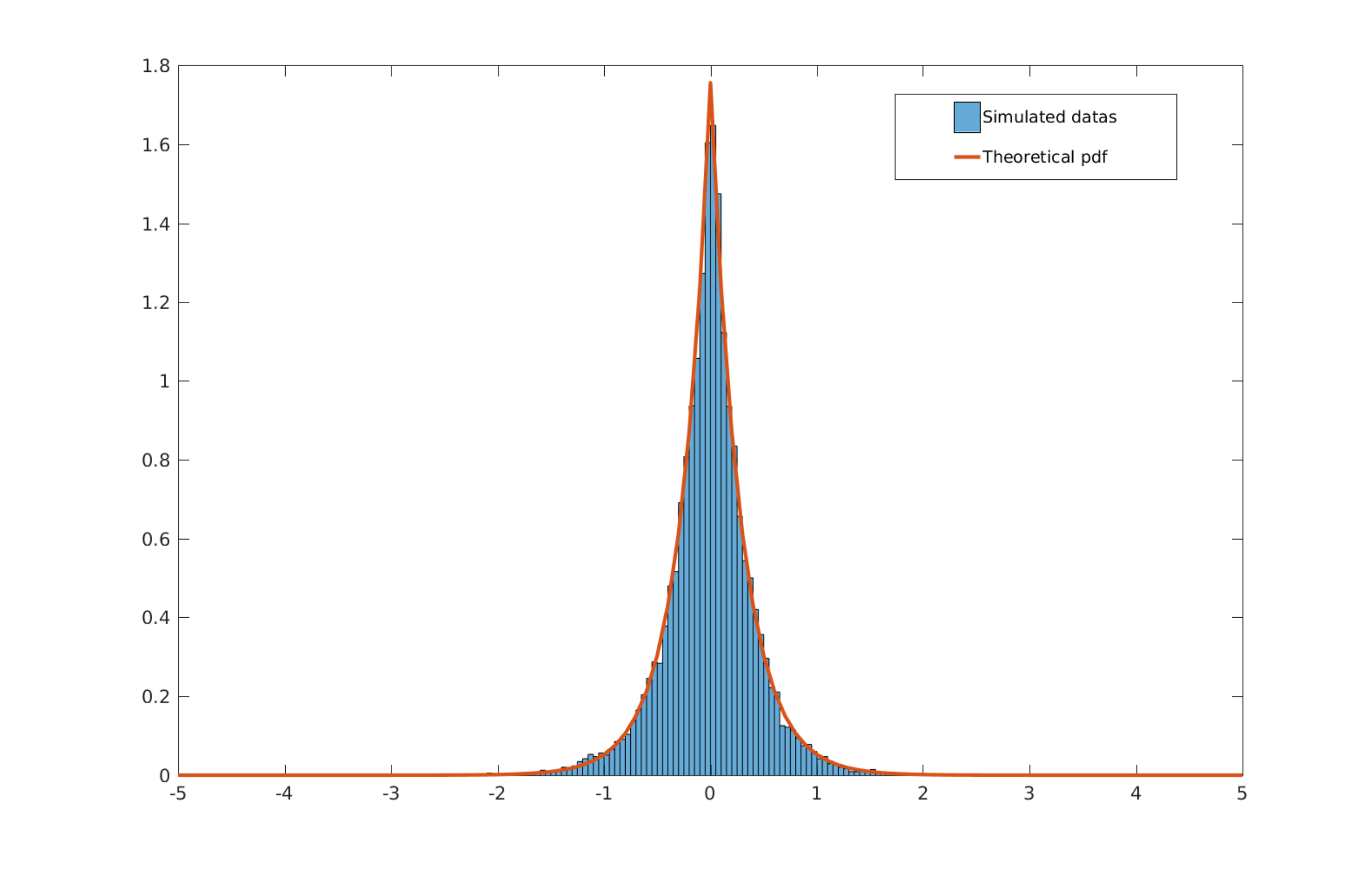}
	\caption{Distribution of the renormalized error and expected limit distribution given by our CLT.}
	\label{fig:tcl1}
\end{figure}
\begin{figure}
	\includegraphics[scale=0.45,trim=50 0 0 0 mm]{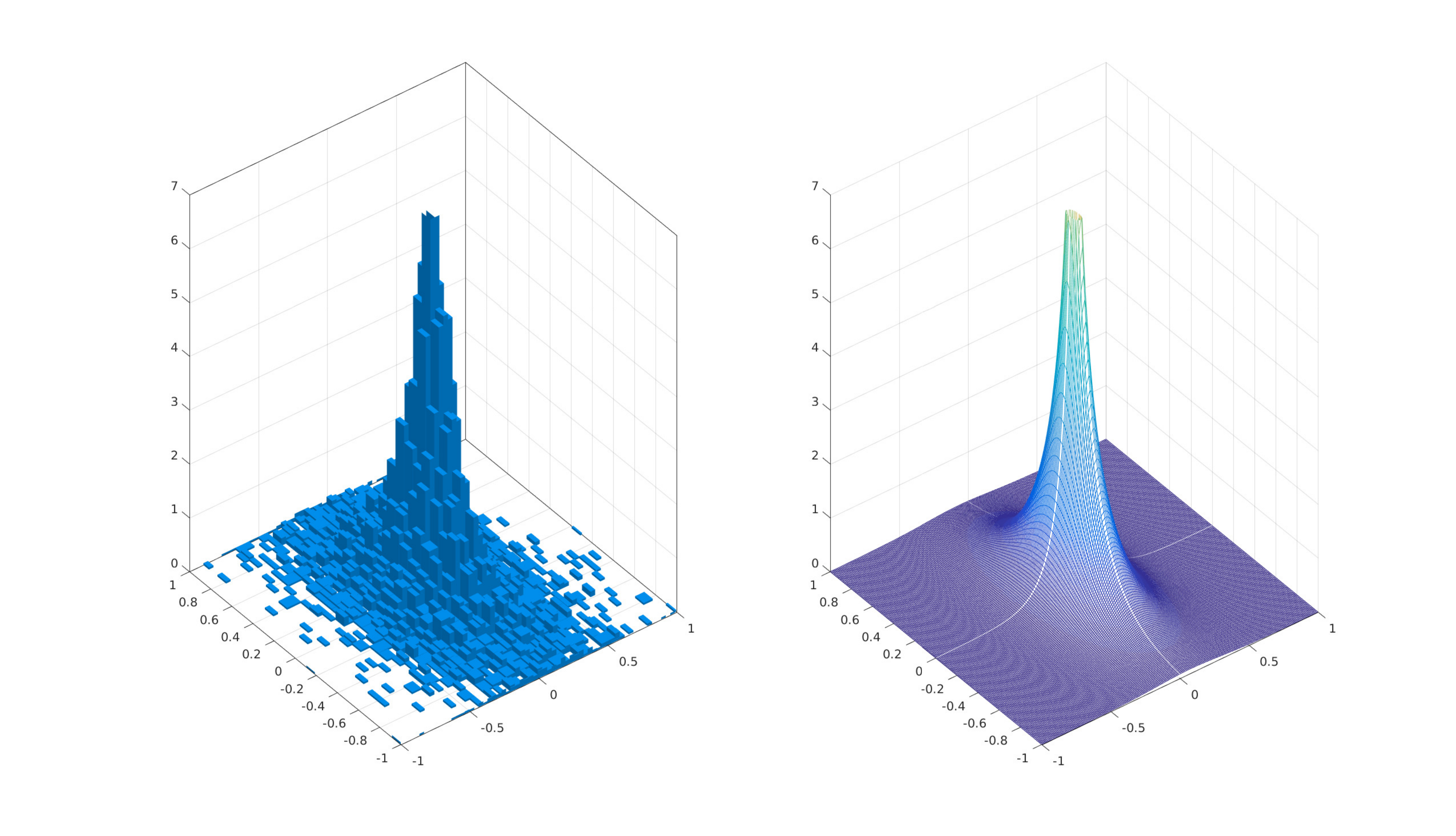}
	\caption{Joint distribution of the renormalized error (left figure) and expected limit distribution (right figure) given by our CLT.}
	\label{fig:tcl2}
\end{figure}
\begin{figure}
	\includegraphics[scale=0.5, trim = -70 50 100 50 mm]{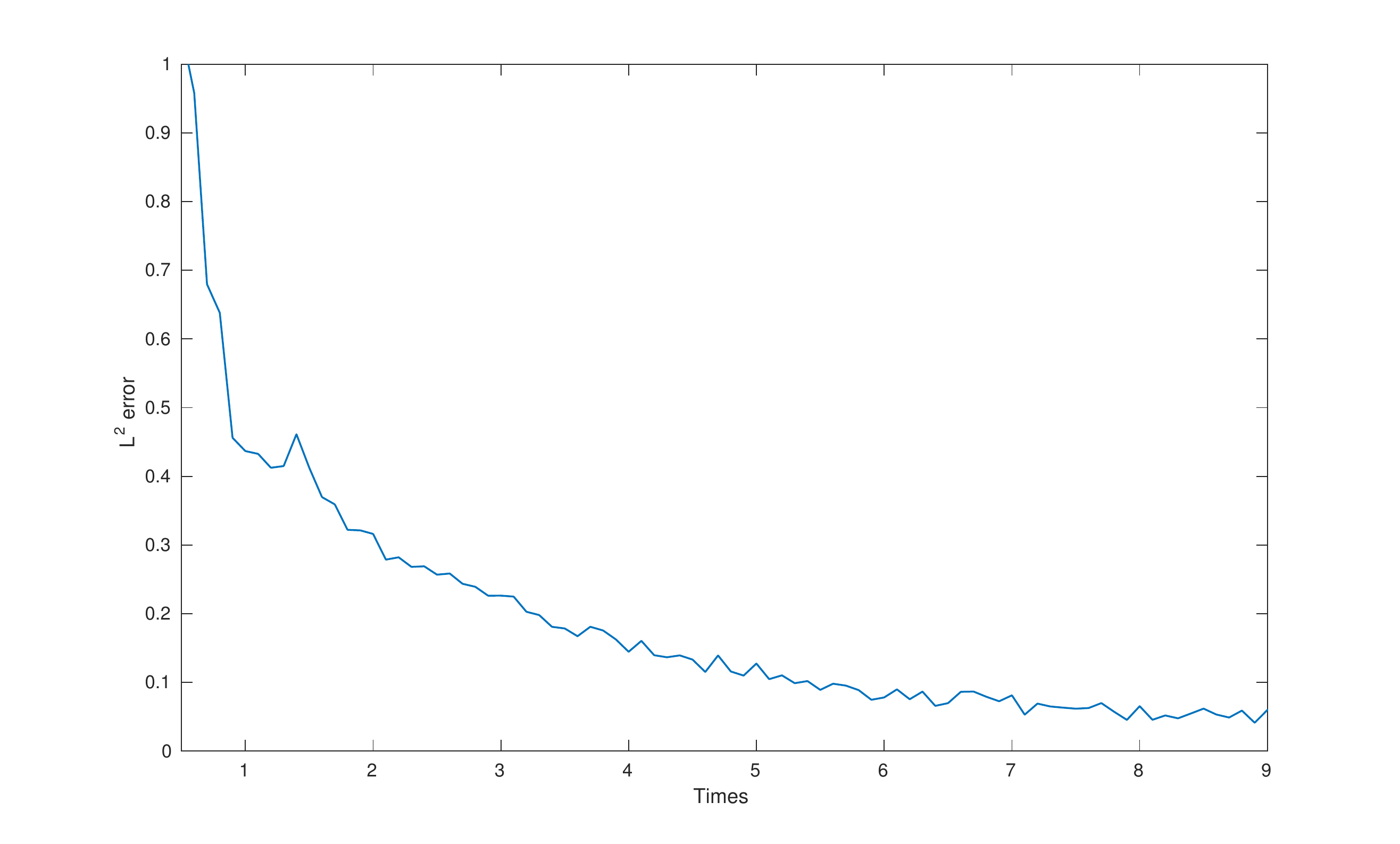}
	\caption{Estimation of the rate of convergence in $L^2$ norm.}
	\label{fig:L2err}
\end{figure}
Another interesting question which could be partially probed by simulation is the study of the behavior of the error in the clonal supercritical case. Figure \ref{fig:closup} shows a kernel estimation (from $10000$ simulation) of the density of $\psi'(\alpha)(e^{\alpha\frac{t}{2}}(A(1,t)-c_{k}N_{t}))$ in the clonal supercritical case ($\theta=0.2<\alpha$). Figure \ref{fig:closup} suggest a totally different behavior with a limit distribution which is asymmetric with respect to $0$. In particular, in the view of the shape of the distribution, one could conjecture that the limit is a skew stable distribution.

\begin{figure}
	\includegraphics[scale=0.5, trim = -70 0 50 0 mm]{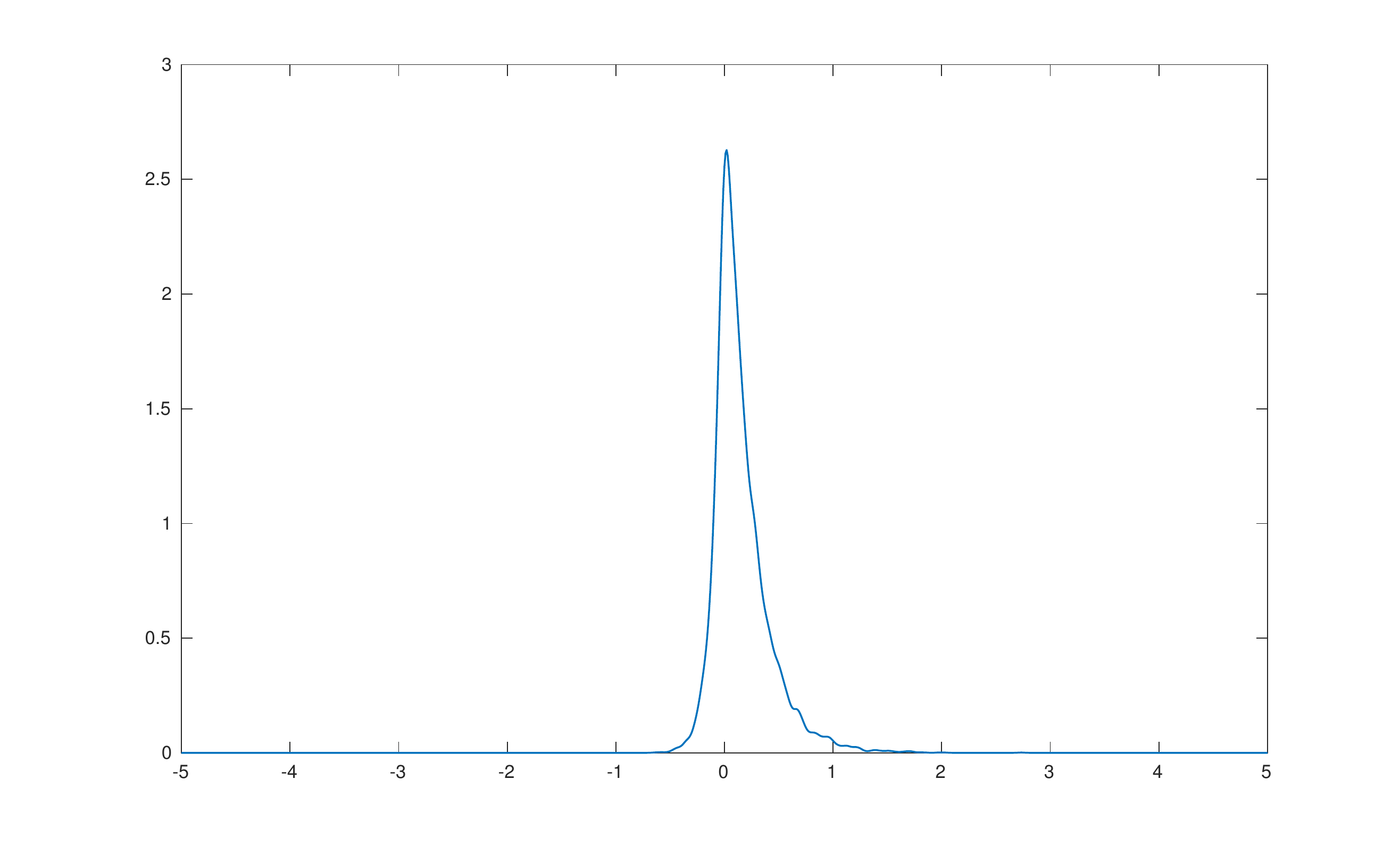}
	\caption{Kernel estimate of the probability density function of the limit distribution in the clonal supercritical case.}
	\label{fig:closup}
\end{figure}
To end this section, let us goes back to one of the motivation of this work. The following discussion dot not claim to be rigorous and is essentially formal. We recall that the Extended Haplotype Homozygosity (EHH) can be used to detect positive selection in a population \cite{sabeti}. In particular, the behavior of the frequency spectrum in this model gives a standard for the behavior of a subpopulation carrying a common allele under neutral evolution. In order to have a rigorous model to describe this phenomenon, we need to introduce a new mutation measure which is different from the one given in Section \ref{sec:models}. We define it directly on the CPP but this could be equivalently defined on the splitting tree. So let $\mathcal{P}$ be a Poisson random measure on $[0,t]\times\mathbb{N}\times\mathbb{R}_{+}$ with intensity measure $ \lambda\otimes C\otimes\lambda$, where $C$ is the counting measure on $\mathbb{N}$, then, for any mutation rate $\theta$ in $\mathbb{R}_{+}$, we define the $\theta$-mutation random measure $\mathcal{N}_{\theta}$ by
\[
\mathcal{N}_{\theta}\left(A\times B\right)=\int_{A\times B\times [0,\theta]}\mathds{1}_{H_{i}>t-a}\mathds{1}_{i<\mathcal{N}_{t}}\mathcal{P}\left(di,da,dx\right),
\]
where, as before, an atom at $(a,i)$ means that the $i$th branch experiences a mutation at time $t-a$. This construction allows to increase the mutation in consistent manner. This allows to model the type of an individual at a distance $x$ (such that the mutation rate is a function of $x$)  from the core haplotype (we refer the reader to \cite{sabeti} for more details). Now, following \cite{CH}, we can define the frequency spectrum at mutation rate $\theta$ by
\[
A^{\theta}(k,t)=\int_{[0,t]\times\mathbb{N}}\mathds{1}_{Z_{0}(i,a)=k}\mathcal{N}_{\theta}(di,da),
\]
where $Z_{0}(i,a)$ is the number of individual at time $t$ carrying the type of the $i$th individual at time $t-a$ (see \cite{CH} for more details). Let us also define $Z_{0}^{\theta}(t)$ the number of individuals carrying the type of the first individual at time $0$ when the mutation measure is given by $\mathcal{N}_{\theta}$.
As expected, the allelic partition of the population becomes thinner as $\theta$ growth.

Now, the definition of the EHH $G_{\theta}(t)$ is the probability that two uniformly sampled individuals in the population have the same type, that is
\[
G_{t}(\theta)=\frac{Z^{\theta}_{0}(t)(Z^{\theta}_{0}(t)-1)+\sum_{k\geq1}k(k-1)A^{\theta}(k,t)}{N_{t}(N_{t}-1)}.
\]
Using that 
\[
N_{t}=Z^{\theta}_{0}(t)+\sum_{k\geq 1}kA^{\theta}(k,t),
\]
this rewrite
\[
G_{t}(\theta)=\frac{(N_{t}-\sum_{k\geq 1}kA^{\theta}(k,t))(N_{t}-\sum_{k\geq 1}kA^{\theta}(k,t)-1)+\sum_{k\geq1}k(k-1)A^{\theta}(k,t)}{N_{t}(N_{t}-1)}.
\]
Finally, using the approximation
\[
(A(k,t))_{k\geq 1}\approx (c_{k})_{k\geq 1}N_{t}
\]
 proposed in this work, one could expect that
\[
G_{t}(\theta)\approx\frac{\sum_{k\geq1}k(k-1)c_{k}}{N_{t}}=\frac{\int_{0}^{\infty}2\theta e^{-\theta x}(W_{\theta}(x)-1)dx}{N_{t}}.
\]
We stress the fact that the above expression make sens only in the clonal subcritical case (in the other cases the integral in note finite). Now, we can look at the accuracy of this approximation in view of numerical simulation. Figure \ref{fig:EHH} shows the value of the EHH (when $\theta$ increase) from a simulation of the model (blue curve) and the one obtained using our approximation (red curve). In view of Figure \ref{fig:EHH}, the approximation seems pretty accurate. In order to be more quantitative, Figure \ref{fig:errorEHH} shows the relative error between the EHH and its approximation for one simulation. This shows that the error, as least for sufficiently large $\theta$, remains under $8$\%. To end, let us highlight that Theorem \ref{thm:cltFin} can be used to give confidence intervals for fixed $\theta$ but in order the construct tests of selection from curves like these of Figure \ref{fig:EHH} one would need to have functional CLT in long time for the process $( (A^{\theta}(k,t)-c^{\theta}_{k}N_{t})_{k\geq1},\ \theta\in\mathbb{R}_{+})$.

\begin{figure}
\includegraphics[scale=0.5, trim = -70 50 100 50 mm]{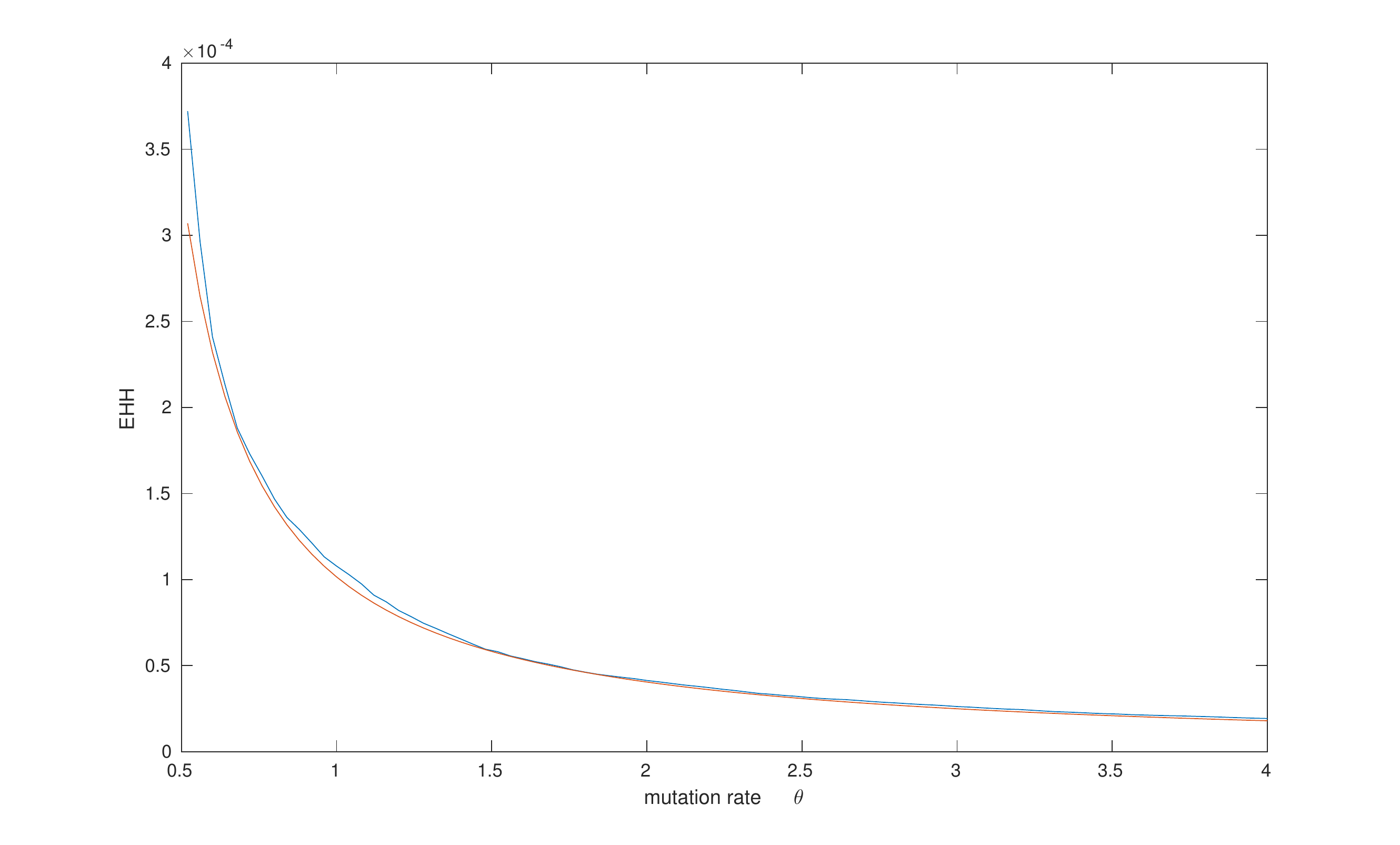}
\caption{Extended Haplotype Homozygosity (EHH) with the given approximation (red curve) and from simulated data (blue curve)}
\label{fig:EHH}
\end{figure}
\begin{figure}
	\includegraphics[scale=0.5, trim = -70 50 100 50 mm]{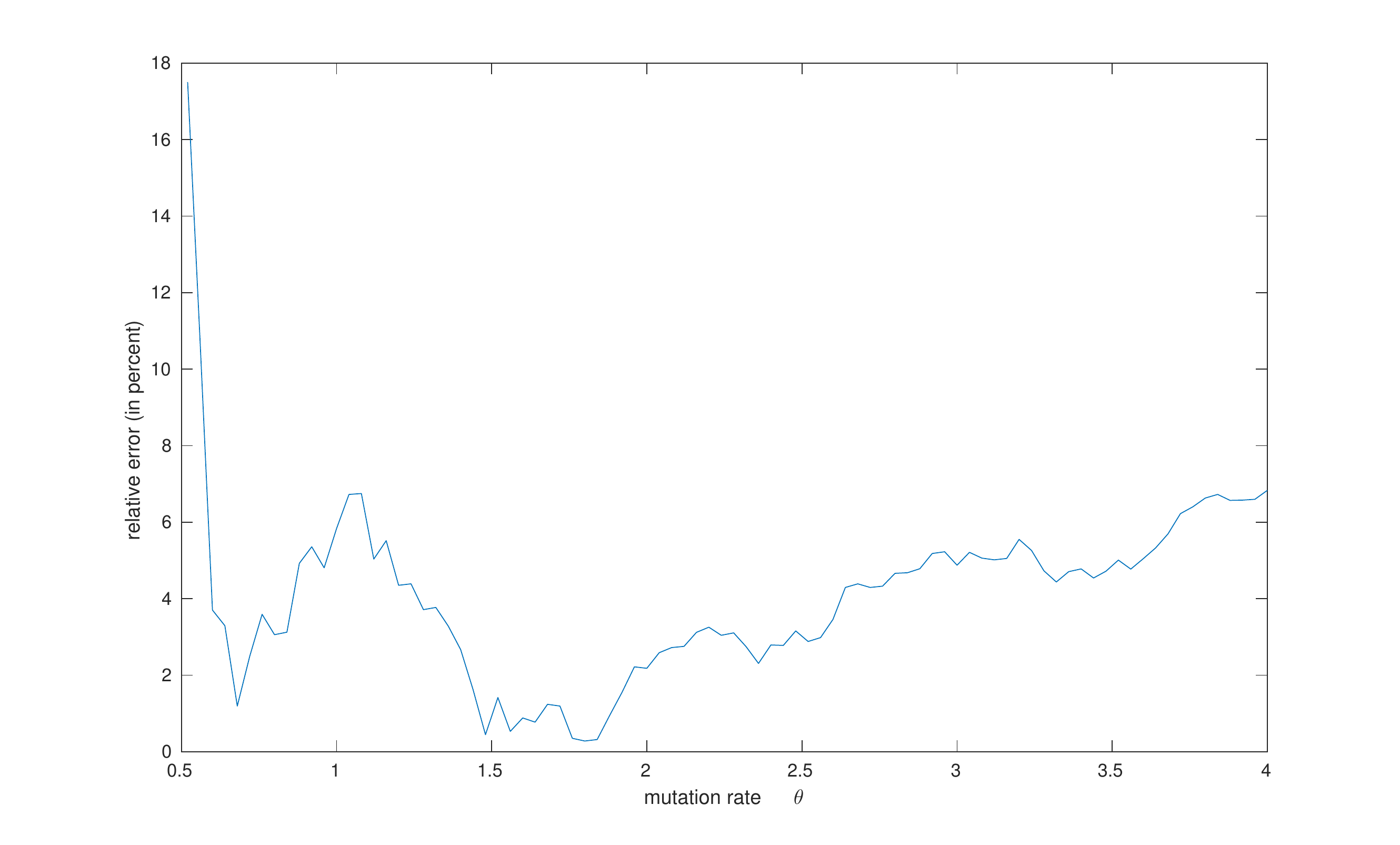}
	\caption{Relative error in the approximation of the EHH }
	\label{fig:errorEHH}
\end{figure}
\appendix
 \section{A bit of renewal theory}
 \label{ssec:renew}
 The purpose of this part is to recall some facts on renewal equations borrowed from \cite{Fel}.
 Let $h:\mathbb{R}\to\mathbb{R}$ be a function bounded on finite intervals with support in $\mathbb{R}_{+}$ and $\Gamma$ a probability measure on $\mathbb{R}_{+}$.
 The equation
 \[
 F(t)=\int_{\mathbb{R}_{+}}F(t-s)\Gamma(ds)+h(t),
 \]
 called a renewal equation, is known to admit a unique solution finite on bounded interval.
 
 Here, our interest is focused on the asymptotic behavior of $F$. We said that the function $h$ is DRI (directly Riemann integrable) if for any $\delta>0$, the quantities
 \[
 \delta\sum_{i=0}^{n}\sup_{t\in[\delta i,\delta(i+1))}f(t)
 \]
 and
 \[
 \delta\sum_{i=0}^{n}\inf_{t\in[\delta i,\delta(i+1))}f(t)
 \]
 converge as $n$ goes to infinity respectively to some real number $I_{sup}^{\delta}$ and $I_{inf}^{\delta}$,
 and
 \[
 \lim\limits_{\delta\to 0}I_{sup}^{\delta}=\lim\limits_{\delta\to 0}I_{inf}^{\delta}<\infty.
 \]
 In the sequel, we use the two following criteria for the DRI property:
 \begin{lem}
 	\label{lem:DRI}
 	Let $h$ a function as defined previously. If $h$ satisfies one of the next two conditions, then $h$ is DRI:
 	\begin{itemize}
 		\item[1.] $h$ is non-negative decreasing and classically Riemann integrable on $\mathbb{R}_{+}$,
 		\item[2.]  $h$ is c\`adl\`ag and bounded by a DRI function.
 	\end{itemize}
 \end{lem}
 We can now state the next result, which is constantly used in the sequel.
 \begin{thm}
 	\label{thm:nonLatticeRen}
 	Suppose that $\Gamma$ is non-lattice, and $h$ is DRI, then
 	\[
 	\lim\limits_{t\to\infty}F(t)=\gamma\int_{\mathbb{R}_{+}}h(s)ds,
 	\]
 	with \[
 	\gamma:=\left(\int_{\mathbb{R}_{+}}s\ \Gamma(ds)\right)^{-1},
 	\]
 	if the above integral is finite, and zero otherwise.

 \end{thm}
 \begin{rem}
 	\label{rem:1}
 	In particular, if we suppose that $\Gamma$ is a measure with mass lower than $1$, and that there exists a constant $\alpha\geq0$ such that
 	\[
 	\int_{\mathbb{R}_{+}}e^{\alpha t}\Gamma(dt)=1,
 	\]
 	then, one can perform the change a measure
 	\[
 	\widetilde{\Gamma}(dt)=e^{\alpha t}\Gamma(dt),
 	\]
 	in order to apply Theorem  \ref{thm:nonLatticeRen} to a new renewal equation to obtain the asymptotic behavior of $F$. (See \cite{Fel} for details).
 	This method is also used in the sequel.
 \end{rem}

\section{Formula for the fourth moment of the error}
\begin{lem}
\label{lem:fourthMoment}
\begin{align*}
\mathbb{E}_{t}&\left[\left(A(k,t)-c_{k}N_{t}\right)^{4} \right]=4\int_{[0,t]}\theta\frac{W(t)}{W(a)}\mathbb{E}_{a}\left[\mathds{1}_{Z_{0}(a)=k}\left(A(k,a)-c_{k}N_{a}\right)^{3}\right]da\\
&+48\int_{[0,t]}\theta\frac{W(t)^{2}}{W(a)^{2}}\left(1-\frac{W(a)}{W(t)}\right)\mathbb{E}_{a}\left[\mathds{1}_{Z_{0}(a)=k}N_{a}A(k,a)\right]\mathbb{E}_{a}\left[\left(c_{k}N_{a}-A(k,a)\right)\right]da\\
&+24\int_{[0,t]}\theta\frac{W(t)^{2}}{W(a)^{2}}\left(1-\frac{W(a)}{W(t)}\right)\mathbb{E}_{a}\left[\mathds{1}_{Z_{0}(a)=k}N_{a}^{2}\right]\mathbb{E}_{a}\left[\left(A(k,a)-c_{k}N_{a}\right)\right]da\\
&+24\int_{[0,t]}\theta\frac{W(t)^{2}}{W(a)^{2}}\left(1-\frac{W(a)}{W(t)}\right)\mathbb{E}_{a}\left[\mathds{1}_{Z_{0}(a)=k}A(k,a)^{2}\right]\mathbb{E}_{a}\left[\left(A(k,a)-c_{k}N_{a}\right)\right]da\\
&+8\int_{[0,t]}\theta\frac{W(t)^{2}}{W(a)^{2}}\left(1-\frac{W(a)}{W(t)}\right)\mathbb{P}_{a}\left(Z_{0}(a)=k\right)\mathbb{E}_{a}\left[\left(A(k,a)-c_{k}N_{a}\right)^{3}\right]da\\
&+48\int_{[0,t]}\theta\frac{W(t)^{2}}{W(a)^{2}}\left(1-\frac{W(a)}{W(t)}\right)\mathbb{E}_{a}\left[\mathds{1}_{Z_{0}(a)=k}A(k,a)\right]\mathbb{E}_{a}\left[\left(A(k,a)-c_{k}N_{a}\right)^{2}\right]da\\
&+72\int_{[0,t]}\theta\frac{W(t)^{3}}{W(a)^{3}}\left(1-\frac{W(a)}{W(t)}\right)^{2}\mathbb{E}_{a}\left[\mathds{1}_{Z_{0}(a)=k}\left(A(k,a)-c_{k}N_{a}\right)\right]\mathbb{E}_{a}\left[\left(A(k,a)-c_{k}N_{a}\right)\right]^{2}da\\
&+72\int_{[0,t]}\theta\frac{W(t)^{3}}{W(a)^{3}}\left(1-\frac{W(a)}{W(t)}\right)^{2}\mathbb{P}_{a}\left(Z_{0}(a)=k\right)\mathbb{E}_{a}\left[\left(A(k,a)-c_{k}N_{a}\right)^{2}\right]\mathbb{E}_{a}\left[A(k,a)-N_{a}c_{k} \right]da\\
&+96\int_{[0,t]}\theta\frac{W(t)^{4}}{W(a)^{4}}\left(1-\frac{W(a)}{W(t)}\right)^{3}\mathbb{P}_{a}\left(Z_{0}(a)=k\right)\mathbb{E}_{a}\left[\left(A(k,a)-c_{k}N_{a}\right)\right]^{3}da+c_{k}^{4}\mathbb{E}_{t}N_{t}^{4}\\
\end{align*}

\end{lem}
\begin{proof}
The proof of this Lemma lies on the calculation of the expectation of each term in the development of
\[
\left(A(k,t)-c_{k}N_{t}\right)^{4}.
\]
To make this, we intensively use the relation \eqref{eq:recFormula} and the method developed in \cite{CH}.
We begin by computing \[
\mathbb{E}_{t}\left[A(k,t)^{4}\right].
\]
Formula \eqref{eq:recFormula} gives us,
\begin{align}
\label{eq:decompInd}
A(k,t)^{4}=&4\int_{[0,t]\times \mathbb{N}}\mathds{1}_{Z^{i}_{0}(a)=k_{i}}\sum_{u_{1:3}=1}^{N^{(t)}_{t-a}}\prod_{\underset{i\neq j}{j=1}}^{3}A^{(u_{j})}(k,a)\mathcal{N}\left(da,di \right)\notag\\
=&4\int_{[0,t]\times \mathbb{N}}\mathds{1}_{Z^{i}_{0}(a)=k}A^{i}(k,a)A^{i}(k,a)A^{i}(k,a)\mathcal{N}(da,di)\notag\\
&+4\int_{[0,t]\times \mathbb{N}}\mathds{1}_{Z^{i}_{0}(a)=k}\sum_{\underset{j_{1}\neq j_{2}\neq j_{3}\neq i}{j_{1},j_{2},j_{3}=1}}^{N^{(t)}_{t-a}}A^{j_{1}}(k,a)A^{j_{2}}(k,a)A^{j_{3}}(k,a)\mathcal{N}(da,di)\notag\\
&+12\int_{[0,t]\times \mathbb{N}}\mathds{1}_{Z^{i}_{0}(a)=k}A^{i}(k,a)A^{i}(k,a)\sum_{j=1, j\neq i}^{N^{(t)}_{t-a}}A^{j}(k,a)\mathcal{N}(da,di)\notag\\
&+4\int_{[0,t]\times \mathbb{N}}\mathds{1}_{Z^{i}_{0}(a)=k}\sum_{j=1, j\neq i}^{N^{(t)}_{t-a}}A^{j}(k,a)^{3}\mathcal{N}(da,di)\notag\\
&+12\int_{[0,t]\times \mathbb{N}}\mathds{1}_{Z^{i}_{0}(a)=k}A^{i}(k,a)\sum_{j_{1},j_{2}=1, j_{1}\neq j_{2}\neq i}^{N^{(t)}_{t-a}}A^{j_{1}}(k,a)A^{j_{2}}(k,a)\mathcal{N}(da,di)\notag\\
&+24\int_{[0,t]\times \mathbb{N}}\mathds{1}_{Z^{i}_{0}(a)=k}A^{i}(k,a)\sum_{j_{1}=1, j_{1}\neq i}^{N^{(t)}_{t-a}}A^{j_{1}}(k,a)A^{j_{1}}(k,a)\mathcal{N}(da,di)\notag\\
&+12\int_{[0,t]\times \mathbb{N}}\mathds{1}_{Z^{i}_{0}(a)=k}\sum_{j_{1},j_{2}=1, j_{1}\neq j_{2}\neq i}^{N^{(t)}_{t-a}}A^{j_{1}}(k,a)^{2}A^{j_{2}}(k,a)\mathcal{N}(da,di).
\end{align}
The decomposition of the sum in form
\[
\sum_{u_{1:3}=1}^{N^{(t)}_{t-a}},
\]
has then been made to distinguish independence properties in our calculation. Actually, as soon as, $i\neq j$, $A^{i}(k,a)$ is independent from $A^{i}(k,a)$ (see \cite{CH} for details).
It is essential to note that the expectation of these integrals with respect to the random measure $\mathcal{N}$ are all calculated thanks to Theorem 3.1 of \cite{CH}.
So, taking the expectation now leads to,
\begin{align*}
\mathbb{E}_{t}\left[A(k,t)^{4}\right]=&
4\int_{[0,t]}\theta\mathbb{E}_{a}\left[N^{(t)}_{t-a}\right]\mathbb{E}_{a}\left[\mathds{1}_{Z_{0}(a)=k}A(k,a)^{3}\right]\theta da\\
&+4\int_{[0,t]}\theta\mathbb{P}_{a}\left(Z_{0}(a)=k\right)\mathbb{E}_{a}\left[\left(N^{(t)}_{t-a}\right)_{(4)}\right]\mathbb{E}_{a}\left[A(k,a)\right]^{3}da\\
&+12\int_{[0,t]}\theta\mathbb{E}_{a}\left[\mathds{1}_{Z_{0}(a)=k}A(k,a)^{2}\right]\mathbb{E}_{a}\left[\left(N^{(t)}_{t-a}\right)_{(2)}\right]\mathbb{E}_{a}\left[A(k,a)\right]da\\
&+4\int_{[0,t]}\theta\mathbb{P}_{a}\left(Z_{0}(a)=k\right)\mathbb{E}_{a}\left[\left(N^{(t)}_{t-a}\right)_{(2)}\right]\mathbb{E}_{a}\left[A(k,a)^{3}\right]da\\
&+12\int_{[0,t]}\theta\mathbb{E}_{a}\left[\mathds{1}_{Z_{0}(a)=k}A(k,a)\right]\mathbb{E}_{a}\left[\left(N^{(t)}_{t-a}\right)_{(3)}\right]\mathbb{E}_{a}\left[A(k,a)\right]^{2}da\\
&+24\int_{[0,t]}\theta\mathbb{E}_{a}\left[\mathds{1}_{Z_{0}(a)=k}A(k,a)\right]\mathbb{E}_{a}\left[\left(N^{(t)}_{t-a}\right)_{(2)}\right]\mathbb{E}_{a}\left[A(k,a)^{2}\right]da\\
&+12\int_{[0,t]}\theta\mathbb{P}_{a}\left(Z_{0}(a)=k\right)\mathbb{E}_{a}\left[\left(N^{(t)}_{t-a}\right)_{(3)}\right]\mathbb{E}_{a}\left[A(k,a)^{2}\right]\mathbb{E}_{a}\left[A(k,a)\right]da.\\
\end{align*}
Using the same method for all the other terms and that, for any positive real number $a$ lower than $t$,
\[
N_{t}=\sum_{i=1}^{N^{(t)}_{t-a}}N^{(i)}_{a},
\]
we get Lemma \ref{lem:fourthMoment} by reassembling similar terms together. The last term is obtained using the geometric distribution of $N_{t}$ under $\mathbb{P}_{t}$.
\end{proof}

\section{Boundedness of the fourth moment}\begin{lem}
We begin the proof of the boundedness of the fourth moment by some estimates.
\label{lem:CompAs}
\begin{equation}
\label{eq:CompAs1}
\tag{i}
\mathbb{E}_{t}\left[\left(A(k,t)-c_{k}N_{t}\right) \right]=\mathcal{O}\left(e^{-(\theta-\alpha)t} \right),
\end{equation}
\begin{equation}
\label{eq:CompAs2}
\tag{ii}
\mathbb{E}_{t}\left[\left(A(k,t)-c_{k}N_{t}\right)^{3} \right]=\mathcal{O}\left(W(t)^{2} \right),
\end{equation}
\begin{equation}
\label{eq:CompAs3}
\tag{iii}
\mathbb{E}_{t}\left[\left(A(k,t)-c_{k}N_{t}\right)^{2} \right]=\mathcal{O}\left(W(t)\right),
\end{equation}
\begin{equation}
\label{eq:CompAs4}
\tag{iv}
\mathbb{E}_{t}N_{t}^{n}=\mathcal{O}(e^{n\alpha t}), \quad n\in\mathbb{N}^{*},
\end{equation}
\begin{equation}
\label{eq:CompAs5}
\tag{v}
\mathbb{P}_{t}\left(Z_{0}(t)=k\right)=\mathcal{O}(e^{(\alpha-\theta)t}).
\end{equation}
\end{lem}

\begin{proof}
Relation \eqref{eq:CompAs1} is easily obtained using the expectation of $N_{t}$ and $A(k,t)$ using  \eqref{eq:Mom1akate}, \eqref{eq:ck} and the behaviour of $W$ provided by Proposition \ref{lem:WComp}. The relation \eqref{eq:CompAs3} has been obtained in the proof of Theorem 6.1 in \cite{CH}.
The two last relations are easily obtained from \eqref{eq:loiNt}, \eqref{eq:loizzero} and Lemma \ref{lem: asyComp}.
The relation \eqref{eq:CompAs2} is obtained using the following estimation,
\[
\left|\mathbb{E}_{t}\left[\left(A(k,a)-c_{k}N_{a}\right)^{3} \right] \right|\leq \mathbb{E}_{t}\left[N_{a}\left(A(k,a)-c_{k}N_{a}\right)^{2} \right].
\]

We begin the proof by computing the r.h.s. of the previous inequality using the same techniques as in Appendix A.
\begin{align*}
\mathbb{E}\left[A(k,t)^{2}N_{t}\right]&=2\int_{0}^{t}\theta\frac{W(t)}{W(a)}\mathbb{E}\left[N_{a}A(k,a)\mathds{1}_{Z_{0}(a)=k}\right]da\\
+&4\int_{0}^{t}\theta\frac{W(t)^{2}}{W(a)^{2}}\left(1-\frac{W(a)}{W(t)}\right)\mathbb{E}\left[N_{a}\mathds{1}_{Z_{0}(a)=k}\right]\mathbb{E}\left[A(k,a)\right]da\\
+&4\int_{0}^{t}\theta\frac{W(t)^{2}}{W(a)^{2}}\left(1-\frac{W(a)}{W(t)}\right)\mathbb{E}\left[A(k,a)\mathds{1}_{Z_{0}(a)=k}\right]\mathbb{E}\left[N_{a}\right]da\\
+&4\int_{0}^{t}\theta\frac{W(t)^{2}}{W(a)^{2}}\left(1-\frac{W(a)}{W(t)}\right)\mathbb{P}_{a}\left(Z_{0}(a)=k\right)\mathbb{E}\left[A(k,a)N_{a}\right]da\\
+&12\int_{0}^{t}\theta\frac{W(t)^{3}}{W(a)^{3}}\left(1-\frac{W(a)}{W(t)}\right)^{2}\mathbb{P}_{a}\left(Z_{0}(a)=k\right)\mathbb{E}\left[A(k,a)\right]\mathbb{E}\left[N_{a}\right]da.
\end{align*}

\begin{align*}
2\mathbb{E}\left[A(k,t)N_{t}^{2}\right]&=2\int_{0}^{t}\theta\frac{W(t)}{W(a)}\mathbb{E}\left[N_{a}^{2}\mathds{1}_{Z_{0}(a)=k}\right]da\\
+&8\int_{0}^{t}\theta\frac{W(t)^{2}}{W(a)^{2}}\left(1-\frac{W(a)}{W(t)}\right)\mathbb{E}\left[N_{a}\mathds{1}_{Z_{0}(a)=k}\right]\mathbb{E}\left[N_{a}\right]da\\
+&4\int_{0}^{t}\theta\frac{W(t)^{2}}{W(a)^{2}}\left(1-\frac{W(a)}{W(t)}\right)\mathbb{P}_{a}\left(Z_{0}(a)=k\right)\mathbb{E}\left[N_{a}^{2}\right]da\\
+&12\int_{0}^{t}\theta\frac{W(t)^{3}}{W(a)^{3}}\left(1-\frac{W(a)}{W(t)}\right)^{2}\mathbb{P}_{a}\left(Z_{0}(a)=k\right)\mathbb{E}\left[N_{a}\right]^{2}da.
\end{align*}
Finally,
\begin{align*}
\mathbb{E}\left[N_{t}\left(A(k,t)-c_{k}N_{t}\right)^{2}\right]&=2\int_{0}^{t}\theta\frac{W(t)}{W(a)}\mathbb{E}\left[N_{a}\left(A(k,a)-c_{k}N_{a}\right)\mathds{1}_{Z_{0}(a)=k}\right]da\\
+&4\int_{0}^{t}\theta\frac{W(t)^{2}}{W(a)^{2}}\left(1-\frac{W(a)}{W(t)}\right)\mathbb{E}\left[N_{a}\mathds{1}_{Z_{0}(a)=k}\right]\mathbb{E}\left[A(k,a)-c_{k}N_{a}\right]da\\\\
+&4\int_{0}^{t}\theta\frac{W(t)^{2}}{W(a)^{2}}\left(1-\frac{W(a)}{W(t)}\right)\mathbb{E}\left[\left(A(k,a)-c_{k}N_{a}\right)\mathds{1}_{Z_{0}(a)=k}\right]\mathbb{E}\left[N_{a}\right]da\\
+&4\int_{0}^{t}\theta\frac{W(t)^{2}}{W(a)^{2}}\left(1-\frac{W(a)}{W(t)}\right)\mathbb{P}_{a}\left(Z_{0}(a)=k\right)\mathbb{E}\left[N_{a}\left(A(k,a)-c_{k}N_{a}\right)\right]da\\
+&12\int_{0}^{t}\theta\frac{W(t)^{3}}{W(a)^{3}}\left(1-\frac{W(a)}{W(t)}\right)^{2}\mathbb{P}_{a}\left(Z_{0}(a)=k\right)\mathbb{E}\left[N_{a}\right]\mathbb{E}\left[A(k,a)-c_{k}N_{a}\right]da\\
+&c_{k}^{2}\mathbb{E}_{t}N_{t}^{3}.
\end{align*}
Now, an analysis similar to the one of Lemma \ref{lem: thridBounded} leads to the result.
\end{proof}
\begin{proof}[Proof of Lemma \ref{lem: thridBounded}]

The ideas of the proof, is to analyses one to one every terms of the expression of
\[
\mathbb{E}_{t}\left[\left(A(k,t)-c_{k}N_{t}\right)^{4} \right],
\]
 given by Lemma \ref{lem:fourthMoment} using Lemma \ref{lem:CompAs} to show that they behave as $\mathcal{O}\left(W(t)^{2} \right)$. Since the ideas are the same for every terms, we just give a few examples.

First of all, we consider
\[
\int_{[0,t]}\frac{W(t)}{W(a)}\mathbb{E}_{a}\left[\mathds{1}_{Z_{0}(a)=k}\left(A(k,a)-c_{k}N_{a}\right)^{3}\right]da.
\]
Using Lemma \ref{lem:CompAs} \eqref{eq:CompAs2}, we have
\begin{multline*}
\int_{[0,t]}\frac{W(t)}{W(a)}\mathbb{E}_{a}\left[\mathds{1}_{Z_{0}(a)=k}\left(A(k,a)-c_{k}N_{a}\right)^{3}\right]da=\mathcal{O}\left(W(t)^{2}\right). \\
\end{multline*}

Now take the term
\[
\int_{[0,t]}\frac{W(t)^{2}}{W(a)^{2}}\mathbb{E}_{a}\left[\mathds{1}_{Z_{0}(a)=k}N_{a}^{2}\right]\mathbb{E}_{a}\left[\left(A(k,a)-c_{k}N_{a}\right)\right]da,
\]
we have from Lemma \ref{lem:CompAs} \eqref{eq:CompAs1} and \eqref{eq:CompAs4},
\begin{multline*}
\int_{[0,t]}\frac{W(t)^{2}}{W(a)^{2}}\mathbb{E}_{a}\left[\mathds{1}_{Z_{0}(a)=k}N_{a}^{2}\right]\mathbb{E}_{a}\left[\left(A(k,a)-c_{k}N_{a}\right)\right]da\leq \int_{[0,t]}\frac{W(t)^{2}}{W(a)^{2}}\mathbb{E}_{a}\left[N_{a}^{2}\right]e^{-(\theta-\alpha)a}da=\mathcal{O}\left(W(t)^{2} \right) .\\
\end{multline*}
Every term in $W(t)$ or $W(t)^{2}$ are treated this way.
Now, we consider the term in $W(t)^{4}$ which is

\[
I:=96\int_{[0,t]}\frac{W(t)^{4}}{W(a)^{4}}\mathbb{P}_{a}\left(Z_{0}(a)=k\right)\mathbb{E}_{a}\left[\left(A(k,a)-c_{k}N_{a}\right)\right]^{3}da+24W(t)^{4}c_{k}^{4},
\]
 since $N_{t}$ is geometrically distributed under $\mathbb{P}_{t}$, and that
\begin{equation}
\label{eq:devNt4}
\mathbb{E}_{t}N_{t}^{4}=24W(t)^{4}-36W(t)^{3}+\mathcal{O}(W(t)^{2}).
\end{equation}
On the other hand, using the law of $Z_{0}(t)$ given by \eqref{eq:loizzero} and the expectation of $A(k,t)$ given by \eqref{eq:Mom1akate} (under $\mathbb{P}_{t}$), we have,
\begin{align*}
96&\int_{[0,t]}\frac{W(t)^{4}}{W(a)^{4}}\mathbb{P}_{a}\left(Z_{0}(a)=k\right)\mathbb{E}_{a}\left[\left(A(k,a)-c_{k}N_{a}\right)\right]^{3}da\\
&=-96W(t)^{4}\int_{0}^{t}\frac{\theta e^{-\theta a}}{W_{\theta}(a)^{2}}\left(1-\frac{1}{W_{\theta}(a)} \right)^{k-1}\left(\int_{0}^{a}\frac{\theta e^{-\theta s}}{W_{\theta}(s)^{2}}\left(1-\frac{1}{W_{\theta}(s)} \right)^{k-1}\ ds\right)^{3} \ da\\
&=-24W(t)^{4}\left(\int_{0}^{t}\frac{\theta e^{-\theta a}}{W_{\theta}(a)^{2}}\left(1-\frac{1}{W_{\theta}(a)} \right)^{k-1}\ da\right)^{4}.
\end{align*}
Finally,
\begin{equation*}
I=24W(t)^{4}\left(\int_{t}^{\infty}\frac{\theta e^{-\theta a}}{W_{\theta}(a)^{2}}\left(1-\frac{1}{W_{\theta}(a)} \right)^{k-1}\ da\right)^{4}=\mathcal{O}\left(W(t)^{4}e^{-4\theta t} \right)=o(1).
\end{equation*}

The last example is the most technical and relies with the term in $W(t)^{3}$, which is, using \eqref{eq:devNt4} and Lemma \ref{lem:fourthMoment},
\begin{align*}
J:=&72\int_{[0,t]}\frac{W(t)^{3}}{W(a)^{3}}\mathbb{E}_{a}\left[\mathds{1}_{Z_{0}(a)=k}\left(A(k,a)-c_{k}N_{a}\right)\right]\mathbb{E}_{a}\left[\left(A(k,a)-c_{k}N_{a}\right)\right]^{2}da\\
&+72\int_{[0,t]}\frac{W(t)^{3}}{W(a)^{3}}\mathbb{P}_{a}\left(Z_{0}(a)=k\right)\mathbb{E}_{a}\left[\left(A(k,a)-c_{k}N_{a}\right)^{2}\right]\mathbb{E}_{a}\left[A(k,a)-N_{a}c_{k} \right]da\\
&-288\int_{[0,t]}\frac{W(t)^{3}}{W(a)^{3}}\mathbb{P}_{a}\left(Z_{0}(a)=k\right)\mathbb{E}_{a}\left[\left(A(k,a)-c_{k}N_{a}\right)\right]^{3}da-36c_{k}^{4}W(t)^{3}.\\
\end{align*}
On the other hand, using the calculus made in the proof of Theorem 6.3 of \cite{CH}, we have
\begin{align*}
\mathbb{E}_{a}&\left[\left(A(k,a)-c_{k}N_{a}\right)^{2} \right]\\
=&4\int_{[0,a]}\frac{W(a)^{2}}{W(s)^{2}}\left(1-\frac{W(s)}{W(a)} \right)\mathbb{P}_{s}\left(Z_{0}(s)=k\right)\mathbb{E}_{a}\left(A(k,s)-c_{k}N_{s}\right)ds\\
&+2\int_{[0,a]}\frac{W(s)}{W(a)}\ \mathbb{E}_{a}\left[\mathds{1}_{Z_{0}(s)=k}\left(A(k,s)-c_{k}N_{s}\right)\right]ds
+c_{k}^{2}W(a)^{2}\left(2-\frac{1}{W(a)} \right).
\end{align*}

Substituting this last expression in $J$ leads to
\begin{align*}
J&=-144\int_{[0,t]}\frac{W(t)^{3}}{W(a)^{3}}\mathbb{E}_{a}\left[\mathds{1}_{Z_{0}(a)=k}\left(A(k,a)-c_{k}N_{a}\right)\right]\int_{[a,\infty]}\frac{\mathbb{P}\left(Z_{0}(a)=k \right)}{W(s)^{2}}\mathbb{E}_{a}\left[\left(A(k,s)-c_{k}N_{s}\right)\right]dsda\\
&+144W(t)^{3}\int_{[0,t]}\frac{1}{W(a)}\ \mathbb{E}_{a}\left[\mathds{1}_{Z_{0}(a)=k}\left(A(k,a)-c_{k}N_{a}\right)\right]\int_{[a,t]}\frac{1}{W(s)^{2}}\mathbb{P}_{s}\left(Z_{0}(s)=k\right)\mathbb{E}_{a}\left[A(k,s)-N_{s}c_{k} \right]da\\
&-144c_{k}^{2}\int_{[0,t]}\frac{W(t)^{3}}{W(a)}\mathbb{P}_{a}\left(Z_{0}(a)=k\right)\mathbb{E}_{a}\left[A(k,a)-N_{a}c_{k} \right]da\\
&+144\int_{[0,t]}\frac{W(t)^{3}}{W(a)^{3}}\mathbb{P}\left(Z_{0}(a)=k \right)\mathbb{E}_{a}\left[A(k,a)-N_{a}c_{k} \right]^{3}da\\
&-288\int_{[0,t]}\frac{W(t)^{3}}{W(a)^{2}}\mathbb{P}_{a}\left(Z_{0}(a)=k\right)\int_{[0,a]}\frac{1}{W(s)}\mathbb{P}_{s}\left(Z_{0}(s)=k\right)\mathbb{E}_{a}\left(A(k,s)-c_{k}N_{s}\right)ds\mathbb{E}_{a}\left[A(k,a)-N_{a}c_{k} \right]da\\
&+72\int_{[0,t]}\frac{W(t)^{3}}{W(a)}\mathbb{P}_{a}\left(Z_{0}(a)=k\right)c_{k}^{2}\left(2-\frac{1}{W(a)} \right)\mathbb{E}_{a}\left[A(k,a)-N_{a}c_{k} \right]da\\
&-288\int_{[0,t]}\frac{W(t)^{3}}{W(a)^{3}}\mathbb{P}_{a}\left(Z_{0}(a)=k\right)\mathbb{E}_{a}\left[\left(A(k,a)-c_{k}N_{a}\right)\right]^{3}da-36c_{k}^{4}W(t)^{3}.\\
\end{align*}
Using many times that,
\begin{align*}
\int_{[0,t]}&\frac{\theta \mathbb{P}\left(Z_{0}(a)=k \right)}{W(s)^{2}}\mathbb{E}_{a}\left[\left(A(k,s)-c_{k}N_{s}\right)\right]ds\\=&-\int_{[0,t]}\frac{\theta e^{-\theta s}}{W_{\theta}(s)^{2}}\left(1-\frac{1}{W_{\theta}(s)} \right)^{k-1}\int_{[s,\infty]}\frac{\theta e^{-\theta u}}{W_{\theta}(u)^{2}}\left(1-\frac{1}{W_{\theta}(u)} \right)^{k-1}duds\\
=&\frac{c_{k}^{2}}{2}-\frac{1}{2}\left(\int_{[t,\infty]}\frac{\theta e^{-\theta s}}{W_{\theta}(s)^{2}}\left(1-\frac{1}{W_{\theta}(s)} \right)^{k-1}ds\right)^{2},\\
\end{align*}
thanks to \eqref{eq:loizzero}, \eqref{eq:Mom1akate}, and \eqref{eq:NtNoCond}, we finally get
\begin{align*}
J=&-144\left(c_{k}^{2}-c_{k}(t)^{2} \right)\int_{[0,t]}\frac{W(t)^{3}}{W(a)^{3}}\mathbb{E}_{a}\left[\mathds{1}_{Z_{0}(a)=k}\left(A(k,a)-c_{k}N_{a}\right)\right]da\\
&+36W(t)^{3}\left(c_{k}^{2}\left(\int_{[t,\infty]}\frac{W(t)^{3}}{W(a)^{3}}\mathbb{E}_{a}\left[A(k,a)-N_{a}c_{k} \right]^{3}da\right)^{2}-\left(\int_{[t,\infty]}\frac{W(t)^{3}}{W(a)^{3}}\mathbb{E}_{a}\left[A(k,a)-N_{a}c_{k} \right]^{3}da\right)^{4}\right)\\
&+144\left(c_{k}-c_{k}(t) \right)^{2}\int_{[0,t]}\frac{W(t)^{3}}{W(a)}\mathbb{E}_{a}\left[A(k,a)-N_{a}c_{k} \right]da\\
&+36W(t)^{3}\left(c_{k}-c_{k}(t) \right)^{4}.\\
\end{align*}

This shows that $J$ is  $\mathcal{O}\left(W(t)^{2} \right)$.
\end{proof}

\bibliographystyle{plain}
\bibliography{biblio.tex}

\begin{thebibliography}{10}

\bibitem{laplace1}
Joseph Abate, Gagan~L Choudhury, and Ward Whitt.
\newblock An introduction to numerical transform inversion and its application
  to probability models.
\newblock In {\em Computational probability}, pages 257--323. Springer, 2000.

\bibitem{laplace2}
Joseph Abate and Ward Whitt.
\newblock A unified framework for numerically inverting laplace transforms.
\newblock {\em INFORMS Journal on Computing}, 18(4):408--421, 2006.

\bibitem{AN}
K.~B. Athreya and P.~E. Ney.
\newblock {\em Branching processes}.
\newblock Dover Publications, Inc., Mineola, NY, 2004.
\newblock Reprint of the 1972 original [Springer, New York; MR0373040].

\bibitem{Ber}
Jean Bertoin.
\newblock The structure of the allelic partition of the total population for
  {G}alton-{W}atson processes with neutral mutations.
\newblock {\em Ann. Probab.}, 37(4):1502--1523, 2009.

\bibitem{CH}
Nicolas Champagnat and Henry Benoit.
\newblock Moments of the frequency spectrum of a splitting tree with neutral
  poissonian mutations.
\newblock {\em Electron. J. Probab.}, 21:34 pp., 2016.

\bibitem{CL1}
Nicolas Champagnat and Amaury Lambert.
\newblock Splitting trees with neutral {P}oissonian mutations {I}: {S}mall
  families.
\newblock {\em Stochastic Process. Appl.}, 122(3):1003--1033, 2012.

\bibitem{CL2}
Nicolas Champagnat and Amaury Lambert.
\newblock Splitting trees with neutral {P}oissonian mutations {II}: {L}argest
  and oldest families.
\newblock {\em Stochastic Process. Appl.}, 123(4):1368--1414, 2013.

\bibitem{CLR}
Nicolas Champagnat, Amaury Lambert, and Mathieu Richard.
\newblock Birth and death processes with neutral mutations.
\newblock {\em Int. J. Stoch. Anal.}, pages Art. ID 569081, 20, 2012.

\bibitem{EV}
Warren~J. Ewens.
\newblock {\em Mathematical population genetics. {I}}, volume~27 of {\em
  Interdisciplinary Applied Mathematics}.
\newblock Springer-Verlag, New York, second edition, 2004.
\newblock Theoretical introduction.

\bibitem{Fel}
William Feller.
\newblock {\em An introduction to probability theory and its applications.
  {V}ol. {II}.}
\newblock Second edition. John Wiley \& Sons, Inc., New York-London-Sydney,
  1971.

\bibitem{GK}
J.~Geiger and G.~Kersting.
\newblock Depth-first search of random trees, and {P}oisson point processes.
\newblock In {\em Classical and modern branching processes ({M}inneapolis,
  {MN}, 1994)}, volume~84 of {\em IMA Vol. Math. Appl.}, pages 111--126.
  Springer, New York, 1997.

\bibitem{Gri}
R.~C. Griffiths and Anthony~G. Pakes.
\newblock An infinite-alleles version of the simple branching process.
\newblock {\em Adv. in Appl. Probab.}, 20(3):489--524, 1988.

\bibitem{H1}
B.~{Henry}.
\newblock {Central limit theorem for supercritical binary homogeneous
  Crump-Mode-Jagers processes}.
\newblock {\em to appear in ESAIM:Probability and Statistics}, November 2016.

\bibitem{L10}
Amaury Lambert.
\newblock The contour of splitting trees is a {L}\'evy process.
\newblock {\em Ann. Probab.}, 38(1):348--395, 2010.

\bibitem{L11}
Amaury Lambert, Lea Popovic, et~al.
\newblock The coalescent point process of branching trees.
\newblock {\em The Annals of Applied Probability}, 23(1):99--144, 2013.

\bibitem{L13}
Amaury Lambert and Pieter Trapman.
\newblock Splitting trees stopped when the first clock rings and vervaat's
  transformation.
\newblock {\em Journal of Applied Probability}, 50(01):208--227, 2013.

\bibitem{Rich}
Mathieu Richard.
\newblock {\em Arbres, Processus de branchement non Markoviens et processus de
  L\'evy}.
\newblock Th\`ese de doctorat, Universit\'e Pierre et Marie Curie, Paris 6.

\bibitem{sabeti}
Pardis~C Sabeti, David~E Reich, John~M Higgins, Haninah~ZP Levine, Daniel~J
  Richter, Stephen~F Schaffner, Stacey~B Gabriel, Jill~V Platko, Nick~J
  Patterson, Gavin~J McDonald, et~al.
\newblock Detecting recent positive selection in the human genome from
  haplotype structure.
\newblock {\em Nature}, 419(6909):832--837, 2002.

\end{thebibliography}
\end{document}